\documentclass{article}
\usepackage{amssymb}
\usepackage{amsfonts}
\usepackage{amsmath}

\newtheorem{theorem}{\noindent Theorem}[section]

\newtheorem{corollary}{\noindent Corollary}[section]
\newtheorem{definition}{\noindent Definition}[section]
\newtheorem{lemma}{\noindent Lemma}[section]

\newtheorem{proposition}{\noindent Proposition}[section]
\newtheorem{remark}{\noindent Remark}[section]
\numberwithin{equation}{section}
\numberwithin{figure}{section}
\newenvironment{proof}[1][Proof]{\noindent\textbf{#1.} }{\ \rule{0.5em}{0.5em}}

\begin{document}

\title{Homogenization of a thermo-poro-elastic medium with two-components
and interfacial hydraulic/thermal exchange barrier}
\author{Abdelhamid Ainouz \\
Dept. of Maths, Univ. Sci. Tech. H. Boumedienne\\
Po Box 32, El-Alia, Bab-Ezzouar, Algiers, Algeria}
\date{}
\maketitle

\begin{abstract}
The paper addresses the homogenization of a micro-model of poroelasticity
coupled with thermal effects for two-constituent media and with imperfect
interfacial contact.The homogenized model is obtained by means of the
two-scale convergence technique. It is shown that the
macro-thermo-poro-elasticity model with double porosity/diffusivity contains
in particular some extra terms accounting for the micro-heterogeneities and
imperfect contact at the local scale. Finally, a corrector result is given
under more regularity assumptions on the data.
\end{abstract}

\section{Introduction\label{s1}}

Materials like sedimentary rocks and living tissues are generally considered
as porous, compressible and elastic. Generally speaking, these porous
materials allows for transport mass of some substance through their open
channels such as liquid or gas. The presence of a fluid in such diffusive
porous materials affects its mechanical responses. Their elasticity
properties are then clearly highlighted by the compression resulting from
the fluid pressure since any release of fluid storage causes shrinkage of
the pore volume. This approach, called poroelasticity, accounts for the
coupling of the pore pressure field with the stresses in the skeleton. It is
first derived by K. Von Terzaghi \cite{terz} in the one-dimensional setting
and later generalized by M. Biot \cite{biot} to the three-dimensional case.
It combines the Hookean classic elasticity theory for the mechanical
response of the solid with the Darcy flow diffusion model for the fluid
transport within the pores. However, coupling thermal effects with
poromechanical processes is of great importance in real-world applications
such as geomechanics, civil engineering, biophysics. For instance, cold
water injection into a hot hydrocarbon reservoir causes changes in porosity
and permeability. These heat transfer processes yield soil deformations of
the geothermal well, see M. C. Su\'{a}rez-Arriaga \cite{s-a}. An other
example is given by the effects of the temperature on concrete (buildings,
bridges,...) which highly influence its strength development and its
durability. It is well known that high temperature causes cracking and
strength loss within the concrete mass, see for instance, W. Khaliq and V.
Kodur \cite{khk}, J. Shi \&\ al. \cite{sllh}, K. Y. Shin \&\ al. \cite{skkcj}%
, G. Weidenfeld \&\ al. \cite{wah}, W. Wang \&\ al \cite{wlll}. In
bio-engineering, there are many important applications of thermal diffusion
processes in biological tissues such as thermoregulation, thermotherapy,
radiotherapy, ... In some situations tissues are often assumed poroelastic
materials such as brain tissues \cite{cvv}, bone \cite{cow}, skin \cite%
{gagnc, xl}, living organs \cite{mm, pbp} and tumors \cite{aln, tcurr}.

The thermoporoelasticity theory aims to describe the thermal/fluid flows and
elastic behavior of heat conductive, porous and elastic media, see for
instance R.W. Zimmerman \cite{zim}, O. Coussy \cite{oc}, A. H. Cheng \cite%
{cheng}. A mathematical model for such conductive and porous materials is a
set of governing equations which consists of the momentum balance, the
conservation of mass and the energy balance equations. The purpose is to
predict deformations, fluid pressure and the temperature under different
internal, external forces and thermal sources. It is the classical Biot's
theory of consolidation processes coupled to thermal stresses.

Generally, mathematical models of flows in porous elastic media assume that
the domain consists of a system of single network where the pores have the
same size. However, in many natural situations such as aggregated soils or
fissured rocks, materials exhibit two (or more) dominant pore scales.
Mathematical modeling of multiple porosity media is the subject of great
research activity since the pioneering work of Barenblatt et al. \cite{bzk}%
. This concept has been applied to many areas of engineering, see for
instance Bai \& al. \cite{ber2, br3}, Berryman and Wang \cite{bw}, Cowin
\cite{cow}, Khalil and Selvadurai \cite{ks}, Straughan \cite{strau} and T.
D. Tran Ngoc \& al. \cite{tlb}, Wilson and Aifantis \cite{wa}. In this
context, it is assumed that there exist two porous structures: one is
related to macro-porosity connected to the pores of the material and the
other to micro-porosity connected to fissures of the skeleton. This causes
different pressure fields in the micropores and macropores. Furthermore, the
main temperature effect on any kind of media (solid, liquid or gas) is to
induce the phenomenon of thermal expansion that is an increase or a
shrinkage in volume. Furthermore, the temperature is also considered
different in each phase. The concept of thermoelasticity with at least two
(or multiple) temperatures was first initiated by Chen, Gurtin and Williams
see \cite{cgw} and further developed by many researchers see for instance
Masters \& al. \cite{mpl}, D. Ie\c{s}an \cite{san}, H.M. Youssef \cite{you}.

The goal of this paper is to derive rigorously, by means of the two-scale
convergence technique, a new fully coupled model of thermoporoelasticity
for biphasic media.\textsl{\ }In particular, the work contains some original
and essential advances in the study of homogenization problems applied to
poroelasticity. Notice that there are many works devoted to the
homogenization in poroelasticity and in thermoporoelasticity. We refer the
reader for instance to \cite{ain1, ain2, ain3, cbhd, eb, meir}. The outline
of the paper is divided into 4 main sections: Firstly, a micro-model of
poroelasticity coupled with thermal effects for two-constituent media is
given in section \ref{blet}. It is taken into account that contact between
these constituents is imperfect so that fluid and heat flows through the
interface is proportional to the jump of the pressure and temperature field,
respectively. Then in section \ref{vfmr}, the weak formulation of the
microscale problem and the main results are given. In section \ref{dhm}, the
homogenization of the micro-model is given with the help of the two-scale
convergence technique. The obtained macro-model for thermoporoelasticity
with double porosity/diffusivity is, at my knowledge, new in the
literature. It is mainly shown that micro-heterogeneities and imperfect
contact at the local scale lead to a Biot/Thermal matrices and zeroth order
term at the macroscale, giving rise to absorption/diffusion term. Finally,
in section \ref{cr}, a corrector result is given under more regularity
assumptions on the data.

\section{Derivation of the micro-model\label{blet}}

In this section we derive the set of thermo-hydro-mechanical equations for
poroelastic materials, for more details see for e.g. O. Coussy \cite{oc}.

\subsection{Linear thermoporoelasticity equations}

Let $\Omega $ be a bounded and smooth domain in $\mathbb{R}^{3}$ occupying
during the time interval $\left[ 0,T\right] $, $T>0$ a saturated and
poroelastic body which is subjected to a given body force per unit volume $%
\mathbf{f}_{0}$ [N=Kg.m$^{\text{-2}}$.s$^{\text{-2}}$], to a source/sink
term $g_{0}$ [Kg.m$^{\text{-3}}$.s$^{\text{-1}}$] and to a heat source $%
h_{0} $ per unit time $\ $[Kg.m$^{\text{-1}}$.s$^{\text{-3}}$]. Let \textbf{$%
u$} [m] denote its displacement, $p$ [Pa=Kg.m$^{\text{-1}}$.s$^{\text{-2}}$]
its pressure, $\theta $ [K] its temperature and $\sigma $ [Pa] its Cauchy
stress tensor. Throughout this paper the volumetric density $\rho $ is taken
for simplicity to be a positive constant [Kg.m$^{\text{-3}}$]. Now, we
introduce the Helmholtz free energy:
\begin{equation}
\mathcal{A}:=\mathcal{W}-S\theta  \label{e1}
\end{equation}%
where $\mathcal{W}$ [J.m$^{\text{-3}}$, \ J=Kg.m$^{\text{2}}$.s$^{\text{-2}}$%
] is the internal energy per unit of volume related to the strain work
density/porosity and $S$ [J.m$^{\text{-3}}$.K$^{\text{-1}}$]\ is the entropy
per unit of volume. The free energy $\mathcal{A}$ is the basic quantity to
define the material. In the theory of poromechanics and under the
infinitesimal transformation, we assume that the internal energy is a
function of the state quantities $e$ and $\phi $:
\begin{equation*}
\mathcal{W}=\mathcal{W}\left( e,\phi \right)
\end{equation*}%
where%
\begin{equation*}
e_{ij}\left( \mathbf{u}\right) =\frac{1}{2}\left( \frac{\partial u_{j}}{%
\partial x_{i}}+\frac{\partial u_{i}}{\partial x_{j}}\right) ,\ \ \mathbf{u}%
=\left( u_{i}\right) ,\ 1\leq i,j\leq 3
\end{equation*}%
is the (linearized) strain tensor [dimensionless] and $\phi $
[dimensionless] is the porosity. Assuming an isentropic process, that is $%
\mathrm{d}S=0$, we get from (\ref{e1}) that%
\begin{equation}
\mathrm{d}\mathcal{A}=\mathrm{d}\mathcal{W}-S\mathrm{d}\theta =\sigma
\mathrm{d}e+p\mathrm{d}\phi -S\mathrm{d}\theta \label{e2}
\end{equation}%
where
\begin{equation*}
\ \sigma _{ij}:=\frac{\partial \mathcal{W}}{\partial e_{ij}},\ p:=\frac{%
\partial \mathcal{W}}{\partial \phi }.
\end{equation*}%
Let us introduce the Legendre transform of the free energy $\mathcal{A}$:%
\begin{equation}
\eta :=\mathcal{A}-\phi p.  \label{e3}
\end{equation}%
It follows then from (\ref{e2}) and (\ref{e3}) that
\begin{equation}
\mathrm{d}\eta =\sigma \mathrm{d}e-\phi \mathrm{d}p-S\mathrm{d}\theta .
\label{e4}
\end{equation}%
Expanding the free energy $\eta $ in a Taylor series around a reference
state $\eta _{0}$ which corresponds to a free strain $e_{0}=0$, a reference
pressure $p_{0}$, a reference temperature $\theta _{0}$ and neglecting all
terms up to the second order we obtain%
\begin{equation}
\begin{array}{r}
\eta =\eta \left( e,p,\theta \right) =\eta _{0}+\left. \frac{\partial \eta }{%
\partial e_{ij}}\right\vert _{0}e_{ij}+\left. \frac{\partial \eta }{\partial
p}\right\vert _{0}\left( p-p_{0}\right) +\left. \frac{\partial \eta }{%
\partial \theta }\right\vert _{0}\left( \theta -\theta _{0}\right) + \\
\  \\
\left. \frac{\partial ^{2}\eta }{\partial e_{ij}\partial p}\right\vert
_{0}\left( p-p_{0}\right) e_{ij}+\left. \frac{\partial ^{2}\eta }{\partial
e_{ij}\partial \theta }\right\vert _{0}\left( \theta -\theta _{0}\right)
e_{ij}+\left. \frac{\partial ^{2}\eta }{\partial p\partial \theta }%
\right\vert _{0}\left( p-p_{0}\right) \left( \theta -\theta _{0}\right) \\
\  \\
+\frac{1}{2}\left. \frac{\partial ^{2}\eta }{\partial e_{ij}\partial e_{kh}}%
\right\vert _{0}e_{kh}e_{ij}+\frac{1}{2}\left. \frac{\partial ^{2}\eta }{%
\partial p^{2}}\right\vert _{0}\left( p-p_{0}\right) ^{2}+\frac{1}{2}\left.
\frac{\partial ^{2}\eta }{\partial \theta ^{2}}\right\vert _{0}\left( \theta
-\theta _{0}\right) ^{2}%
\end{array}
\label{e5}
\end{equation}%
where (and in the sequel) summation over repeated indices is used. Assuming
without no loss of generality that the energy function $\eta $ presents an
equilibrium point at $\eta _{0}=0$, that is
\begin{equation*}
\left. \frac{\partial \eta }{\partial e_{ij}}\right\vert _{0}=\left. \frac{%
\partial \eta }{\partial p}\right\vert _{0}=\left. \frac{\partial \eta }{%
\partial \theta }\right\vert _{0}=0,
\end{equation*}%
equation (\ref{e5}) reduces then to
\begin{equation*}
\begin{array}{r}
\eta =\frac{1}{2}e_{ij}\left. \frac{\partial ^{2}\eta }{\partial
e_{ij}\partial e_{kh}}\right\vert _{0}e_{kh}+e_{ij}\left. \frac{\partial
^{2}\eta }{\partial e_{ij}\partial p}\right\vert _{0}\left( p-p_{0}\right) +
\\
\\
e_{ij}\left. \frac{\partial ^{2}\eta }{\partial e_{ij}\partial \theta }%
\right\vert _{0}\left( \theta -\theta _{0}\right) +\left. \frac{\partial
^{2}\eta }{\partial p\partial \theta }\right\vert _{0}\left( p-p_{0}\right)
\left( \theta -\theta _{0}\right) + \\
\\
\frac{1}{2}\left. \frac{\partial ^{2}\eta }{\partial p^{2}}\right\vert
_{0}\left( p-p_{0}\right) ^{2}+\frac{1}{2}\left. \frac{\partial ^{2}\eta }{%
\partial \theta ^{2}}\right\vert _{0}\left( \theta -\theta _{0}\right) ^{2}%
\end{array}%
\end{equation*}%
which can be rewritten in a more simplified form:
\begin{equation}
\begin{array}{r}
\eta =\frac{1}{2}\left( e:\mathbf{A}e\right) -\left( B:e\right) \left(
p-p_{0}\right) -\left( D:e\right) \left( \theta -\theta _{0}\right) \\
\\
+\alpha \left( p-p_{0}\right) \left( \theta -\theta _{0}\right) -\frac{1}{2N}%
\left( p-p_{0}\right) ^{2}-\frac{\upsilon }{2\theta _{0}}\left( \theta
-\theta _{0}\right) ^{2}%
\end{array}
\label{e6}
\end{equation}%
where
\begin{eqnarray}
a_{ijkh} &=&\left. \frac{\partial ^{2}\eta }{\partial e_{ij}\partial e_{kh}}%
\right\vert _{0},\hspace{0.1cm}b_{ij}=\left. -\frac{\partial ^{2}\eta }{%
\partial e_{ij}\partial p}\right\vert _{0},\hspace{0.1cm}d_{ij}=\left. -%
\frac{\partial ^{2}\eta }{\partial e_{ij}\partial \theta }\right\vert _{0}
\label{e7} \\
\ \alpha &=&-\left. \frac{\partial ^{2}\eta }{\partial p\partial \theta }%
\right\vert _{0},\ N=\left. -\frac{\partial ^{2}\eta }{\partial p^{2}}%
\right\vert _{0}^{-1},\ \upsilon =-\theta _{0}\left. \frac{\partial ^{2}\eta
}{\partial \theta ^{2}}\right\vert _{0}  \label{e8}
\end{eqnarray}%
and
\begin{equation*}
E:F=e_{ij}f_{ij}=\sum_{i,j=1}^{3}e_{ij}f_{ij},\ E=\left( e_{ij}\right)
_{1\leq i,j\leq 3},\ F=\left( f_{ij}\right) _{1\leq i,j\leq 3}.
\end{equation*}%
In (\ref{e7}), $\mathbf{A}=\left( a_{ijkh}\right) $ [Kg.m$^{\text{-1}}$.s$^{%
\text{-2}}$] is the (fourth-rank) elasticity stiffness tensor, $B=\left(
b_{ij}\right) $ [dimensionless] the symmetric stress-pressure tensor,
expressing the change in porosity to the strain variation when pressure and
temperature are kept constant and $D$ is the thermal dilation (symmetric)
tensor related to the solid deformation by the following expression:
\begin{equation*}
D=\left( d_{ij}\right) ,\ \ d_{ij}=\left. -\frac{\partial ^{2}\eta }{%
\partial e_{ij}\partial e_{kh}}\frac{\partial e_{kh}}{\partial \theta }%
\right\vert _{0}=a_{ijkh}\gamma _{kh}
\end{equation*}%
where
\begin{equation*}
\gamma _{kh}=-\frac{\partial e_{kh}}{\partial \theta }\ \ \ \text{[K}^{\text{%
-1}}\text{]}
\end{equation*}%
In (\ref{e8}), $\alpha $ [K$^{\text{-1}}$] expresses the volumetric thermal
dilation coefficient with respect to the pore pressure. Furthermore $N$ [Kg.m%
$^{\text{-1}}$.s$^{\text{-2}}$] is the inverse of the compressibility
coefficient so it refers to a modulus relating the pressure $p$ linearly to
the porosity variation $\phi $ when the volumetric dilation is kept zero .
Finally $\upsilon $ [Kg.m$^{\text{-1}}$.s$^{\text{-2}}$.K$^{\text{-1}}$] is
the volumetric heat capacity. From (\ref{e4}) and (\ref{e6}) we deduce the
constitutive equations:%
\begin{eqnarray}
\sigma &=&\partial _{e}\eta =\mathbf{A}e-\left( p-p_{0}\right) B-\left(
\theta -\theta _{0}\right) D,  \label{e9} \\
\phi &=&-\partial _{p}\eta =B:e+\frac{1}{N}\left( p-p_{0}\right) +\alpha
\left( \theta -\theta _{0}\right) ,  \label{e10} \\
S &=&-\partial _{\theta }\eta =D:e+\alpha \left( p-p_{0}\right) +\frac{%
\upsilon }{\theta _{0}}\left( \theta -\theta _{0}\right) .  \label{e11}
\end{eqnarray}%
Equation (\ref{e9}) is well-known in the literature as the Duhamel--Neumann
relation.

In the framework of consolidation assumption, the inertia effects are
neglected, that is $\rho \partial _{tt}^{2}\mathbf{u}\simeq 0$. In this
case, the conservation of linear momentum equation reads in its differential
form as%
\begin{equation}
\mathrm{div}\sigma +\mathbf{f}_{0}=\mathbf{0}.  \label{e12}
\end{equation}%
Notice that
\begin{equation*}
\sigma _{ij}^{%
{{}^\circ}%
}\left( \mathbf{u}\right) =\sigma _{ij}\left( \mathbf{u}\right) -\left(
b_{ij}\left( p-p_{0}\right) +d_{ij}\left( \theta -\theta _{0}\right) \right)
=a_{ijkh}e_{kh}\left( \mathbf{u}\right)
\end{equation*}%
is the effective stress tensor and equation (\ref{e12}) becomes then
\begin{equation}
-\mathrm{div}\left( \mathbf{A}e\left( \mathbf{u}\right) \right) +B\nabla
p+D\nabla \theta =\mathbf{f}_{0}.  \label{e13}
\end{equation}%
In the case of homogeneous and isotropic materials, the phenomenological
tensors $\mathbf{A}$, $B$ and $D$ take the simplified forms:%
\begin{eqnarray}
&&a_{ijkh}=\lambda \delta _{ij}\delta _{kh}+\mu \left( \delta _{ik}\delta
_{jh}+\delta _{ih}\delta _{jk}\right) ,\ b_{ij}=\beta \delta _{ij},
\label{e14} \\
&&d_{ij}=\left( 3\lambda +2\mu \right) \hat{\gamma}\delta _{ij},\ \gamma
_{kh}=\hat{\gamma}\delta _{kh}  \label{e15}
\end{eqnarray}%
where $\left( \delta _{ij}\right) $ is the Kr\"{o}necker symbol, $\beta $
the Biot-Willis coefficient \cite{bw1},\ $\hat{\gamma}$ [K$^{\text{-1}}$] is
well-known as the thermal expansion coefficient and $\lambda $ [Pa], $\mu $
[Pa] are the Lam\'{e}'s constants which are related to the Young's modulus $%
E $ [Pa] and the Poisson's ration $\nu $ (dimensionless) through the
expression:%
\begin{equation*}
\lambda =E\frac{\nu }{\left( 1-2\nu \right) \left( 1+\nu \right) },\ \mu =E%
\frac{\nu }{2\left( 1+\nu \right) }.
\end{equation*}%
The effective stress is then related to $e$ through the linear Hooke's
constitutive law:%
\begin{equation*}
\sigma _{ij}^{%
{{}^\circ}%
}=2\mu e_{ij}+\lambda e_{kk}\delta _{ij}.\
\end{equation*}%
The Biot-Willis coefficient $\beta $ expresses, at constant fluid pressure,
the ratio of the volume of fluid squeezed out of a solid to total volume
change for elastic deformation.\ The coefficient $\hat{\gamma}$ refers to
the fractional change in volume per degree of temperature change. The
coefficient $2\lambda +3\mu $ is the bulk modulus which relates the
volumetric dilation $e_{ii}$ to $\sigma $ at a constant pore pressure.

If the porous solid is saturated by a compressible fluid whose mass density $%
\rho _{0}$ assumed to be constant, then the equation of continuity for the
flow fluid is given by
\begin{equation}
\partial _{t}\phi +\mathrm{div}\mathbf{v}=\frac{g_{0}}{\rho _{0}}
\label{e16}
\end{equation}%
where $\mathbf{v}$ [m. s$^{\text{-1}}$] is the velocity of the total amount
of the fluid content. For laminar flow, the fluid flux is related to the
fluid velocity $\mathbf{v}$ through the Darcy's law (neglecting the gravity
force $\mathbf{g}$):
\begin{equation}
\mathbf{v}=\frac{-k_{f}}{\mu _{d}}\nabla p  \label{e17}
\end{equation}%
where $k_{f}$ is the intrinsic permeability coefficient [m$^{\text{2}}$] and
$\mu _{d}$ [Pa.s] is the dynamic viscosity. The negative sign in the Darcy
law is needed because fluids flow from high pressure to low one, opposite to
the direction of the pressure gradient. Using (\ref{e10}), (\ref{e16}) and (%
\ref{e17}) we get%
\begin{equation}
\partial _{t}\left( \frac{1}{N}p+B:e\left( \mathbf{u}\right) +\alpha \theta
\right) -\mathrm{div}\left( \frac{k_{f}}{\mu _{d}}\nabla p\right) =\frac{%
g_{0}}{\rho _{0}}.  \label{e18}
\end{equation}

Similarly, the classical Fourier's law relates heat flux vector $\mathbf{q}$
[Kg.s$^{\text{-3}}$] with temperature gradient $\nabla \theta $ by the
equation%
\begin{equation*}
\mathbf{q}=-\lambda _{0}\nabla \theta
\end{equation*}%
where $\lambda _{0}$ [kg.m.s$^{\text{-3}}$.K$^{\text{-1}}$] is the thermal
conductivity. Using the first and second laws of thermodynamics \cite{fre},
the energy balance equation reads as follows:%
\begin{equation}
\partial _{t}\mathcal{W}=\sigma :e\left( \partial _{t}\mathbf{u}\right)
+p\partial _{t}\phi -\mathrm{div}\mathbf{q}+h_{0}  \label{e19}
\end{equation}%
where the left hand side represents the instantaneous change in storage
energy in the body, the first term of the right hand side is heat entering
the material and the last term stands for energy generation in the body.
From (\ref{e4}), differentiating $\eta $ with respect to $t$ and taking into
account the fact that $S=-\partial _{\theta }\eta $, it follows that%
\begin{equation}
\partial _{t}\eta =\sigma :e\left( \partial _{t}\mathbf{u}\right) -\phi
\partial _{t}p-S\partial _{t}\theta .  \label{e20}
\end{equation}%
But, using (\ref{e1}) and (\ref{e3}), we also have
\begin{equation}
\partial _{t}\eta =\partial _{t}\mathcal{W}-\phi \partial _{t}p-p\partial
_{t}\phi -S\partial _{t}\theta -\theta \partial _{t}S.  \label{e21}
\end{equation}%
Equating (\ref{e20}) and (\ref{e21}), equation (\ref{e19}) is rewritten as%
\begin{equation}
\theta \partial _{t}S=-\mathrm{div}\mathbf{q}+h_{0}.  \label{e22}
\end{equation}%
Assuming small variations of temperature so that one can replace $\theta $
by $\theta _{0}$, whenever required, and inserting (\ref{e11}) into (\ref%
{e22}) yield to the following equation:
\begin{equation}
\theta _{0}\partial _{t}\left\{ D:e\left( \mathbf{u}\right) +\alpha p+\frac{%
\upsilon }{\theta _{0}}\theta \right\} =\mathrm{div}\left( \lambda
_{0}\nabla \theta \right) +h_{0}.  \label{e23}
\end{equation}%
Equations (\ref{e13}), (\ref{e18}) and (\ref{e23}) are then the fundamental
equations of linear thermoporoelasticity. Denoting
\begin{eqnarray*}
\gamma &=&\left( 3\lambda +2\mu \right) \hat{\gamma},\ \phi _{0}=\frac{1}{N}%
,\ \kappa =\mu _{d}^{-1}k_{f},\  \\
g &=&\rho ^{-1}g_{0},\ \hat{\lambda}=\theta _{0}^{-1}\lambda _{0},\ c=\theta
_{0}^{-1}\upsilon ,\ h=\theta _{0}^{-1}h_{0}
\end{eqnarray*}%
and taking into account (\ref{e14}), (\ref{e15}), Equations (\ref{e13}), (%
\ref{e18}) and (\ref{e23}) are reduced to the following system:%
\begin{equation}
\left\{
\begin{array}{l}
-\left( \lambda +\mu \right) \nabla \left( \mathrm{div}\mathbf{u}\right)
-\mu \Delta \mathbf{u}+\beta \nabla p+\gamma \nabla \theta =\mathbf{f}_{0},
\\
\\
\partial _{t}\left( \phi _{0}p+\beta \mathrm{div}\mathbf{u}+\alpha \theta
\right) -\kappa \Delta p=g, \\
\\
\partial _{t}\left( c\theta +\gamma \mathrm{div}\mathbf{u}+\alpha p\right) -%
\hat{\lambda}\Delta \theta =h.%
\end{array}%
\right.  \label{e24}
\end{equation}

\begin{remark}
Note that if we let $\gamma $ and $\alpha $ to be negligible then the system
\ (\ref{e24}) decouples to the classical Biot system \cite{biot}:
\begin{equation*}
\left\{
\begin{array}{l}
-\mathrm{div}\sigma =-\left( \lambda +\mu \right) \nabla \mathrm{div}\left(
\mathbf{u}\right) -\mu \Delta \mathbf{u}+\beta \nabla p=\mathbf{f}_{0}, \\
\  \\
\partial _{t}\left( \phi p+\beta \mathrm{div}\mathbf{u}\right) -\kappa
\Delta p_{m}=g,%
\end{array}%
\right.
\end{equation*}%
together with the single heat equation
\begin{equation*}
c\partial _{t}\theta -\hat{\lambda}\Delta \theta =h.
\end{equation*}%
On the other hand, neglecting $\beta $ and $\alpha $, the system (\ref{e24})
decouples to the thermoelasticity system:%
\begin{equation*}
\left\{
\begin{array}{l}
-\mathrm{div}\sigma =-\left( \lambda +\mu \right) \nabla \mathrm{div}\left(
\mathbf{u}\right) -\mu \Delta \mathbf{u}+\gamma \nabla \theta =\mathbf{f}_{0},
\\
\ \  \\
\partial _{t}\left( c\theta +\gamma \mathrm{div}\mathbf{u}\right) -\hat{%
\lambda}\Delta \theta =h%
\end{array}%
\right.
\end{equation*}%
with the single pressure diffusion equation
\begin{equation*}
\phi \partial _{t}p-\kappa \Delta p=g.
\end{equation*}
\end{remark}

\subsection{A mathematical model for two-components materials}

In this subsection we shall derive from (\ref{e24}) the mathematical model
of thermoporoelasticity for isotropic materials made of two constituents,
namely matrix and inclusions and which are in imperfect interfacial contact.
The geometrical setting is described as follows: Let $\Omega $ be a bounded
and a smooth domain of $\mathbb{R}^{3}$. We assume that $\Omega $ is divided
into two open sets $\Omega _{1}$ and $\Omega _{2}$ such that $\Omega =\Omega
_{1}\cup \overline{\Omega }_{2}$ and $\partial \Omega _{2}\cap \partial
\Omega =\emptyset $. The medium occupying $\Omega $ is composed of two
poro-elastic solids $\Omega _{1}$ and $\Omega _{2}$ separated by the
interface $\Sigma :=\partial \Omega _{2}$. According to (\ref{e24}), the
local description is given by the following system: for each phase $m=1,2$
corresponding to the material $\Omega _{m}$
\begin{equation}
\left\{
\begin{array}{l}
-\mathrm{div}\sigma _{m}=\mathbf{f}_{m}, \\
\  \\
\sigma _{m}=-\left( \lambda _{m}+\mu _{m}\right) \nabla \mathrm{div}\left(
\mathbf{u}_{m}\right) -\mu _{m}\Delta \mathbf{u}_{m}+\beta _{m}\nabla
p_{m}+\gamma _{m}\nabla \theta _{m}, \\
\  \\
\partial _{t}\left( \phi _{m}p_{m}+\beta _{m}\mathrm{div}\mathbf{u}%
_{m}+\alpha _{m}\theta _{m}\right) -\kappa _{m}\Delta p_{m}=g_{m}, \\
\  \\
\partial _{t}\left( c_{m}\theta _{m}+\gamma _{m}\mathrm{div}\mathbf{u}%
_{m}+\alpha _{m}p_{m}\right) -\hat{\lambda}_{m}\Delta \theta _{m}=h_{m}.%
\end{array}%
\right.  \label{e25}
\end{equation}%
For piecewise homogeneous media, the system (\ref{e25}) is complemented by
interface, boundary and initial conditions. They read as follows: on the
interface $\Sigma $
\begin{subequations}
\label{e26}
\begin{align}
& \mathbf{u}_{1}=\mathbf{u}_{2},  \label{e26a} \\
& \mathcal{\sigma }_{1}\cdot \mathbf{n}=\mathcal{\sigma }_{2}\cdot \mathbf{n}%
,  \label{e26b} \\
& \kappa _{1}\nabla p_{1}\cdot \mathbf{n}=\kappa _{2}\nabla p_{2}\cdot
\mathbf{n},  \label{e26c} \\
& \hat{\lambda}_{1}\nabla \theta _{1}\cdot \mathbf{n}=\hat{\lambda}%
_{2}\nabla \theta _{2}\cdot \mathbf{n},  \label{e26d} \\
& \kappa _{1}\nabla p_{1}\cdot \mathbf{n}=-\varsigma \left(
p_{1}-p_{2}\right) ,  \label{e26e} \\
& \hat{\lambda}_{1}\nabla \theta _{1}\cdot \mathbf{n}=-\omega \left( \theta
_{1}-\theta _{2}\right)  \label{e26f}
\end{align}%
where $\mathbf{n}$ is the unit normal vector on $\Sigma $ pointing outwards
to $\Omega _{2}$. In (\ref{e26e}) $\varsigma $ [kg$^{\text{-1}}$.m$^{\text{2}%
}$.s] is the interfacial hydraulic permeability and in (\ref{e26f}) $\omega $
[kg.s$^{\text{-3}}$.K$^{\text{-1}}$] is the interface thermal conductance.
Conditions (\ref{e26a})-(\ref{e26d}) are the continuity of the
displacements, of the normal stresses and of the normal fluxes (hydraulic
and thermal). It is assumed that the hydraulic/thermal contact between these
two materials is imperfect, so that the fluxes are related to the jump of
pressures and temperatures, see (\ref{e26e}) and (\ref{e26f}). For example
the case $\varsigma =\infty $ corresponds to a perfect hydraulic contact, so
that $p_{1}=p_{2}$ across the interface and the pressure is continuous. If $%
\varsigma =0$ there is no hydraulic contact across the interface, yielding
no motion of the fluid relative to the solid, that is, $\kappa _{1}\nabla
p_{1}\cdot \mathbf{n}=\kappa _{2}\nabla p_{2}\cdot \mathbf{n}=0$. In the
literature (\ref{e26e}) is known as the Deresiewicz-Skalak condition \cite%
{ds} and (\ref{e26f}) as the Newton's cooling law \cite{cj}. On the exterior
boundary $\Gamma $ we assume the following homogeneous Dirichlet boundary
conditions:
\end{subequations}
\begin{equation}
\mathbf{u}_{1}=0\text{, }p_{1}=0,\ \theta _{1}=0.  \label{e27}
\end{equation}%
Finally the initial conditions are given as follows:%
\begin{equation}
\mathbf{u}_{m}\left( 0,x\right) =0\text{, }p_{m}\left( 0,x\right) =0,\
\theta _{m}\left( 0,x\right) =0,\ x\in \Omega _{m}.  \label{e28}
\end{equation}%
In summary, the system (\ref{e25})-(\ref{e28}) is the complete set of
equations for thermoporoelastic media with two-components.\

In this work we shall study a model consisting of "very" small inclusions
embedded in a matrix, so we have to introduce a small and dimensionless
parameter expressing the ratio between the local scale of the inclusions and
the macroscopic scale of the matrix. This will be done in the next
subsection.

\subsection{Scaling\label{sec1}}

We consider a poroelastic composite of dimension $\mathcal{O}\left(
L^{3}\right) $ where $L$ [m] is the characteristic length of the medium at
the macroscopic scale. We assume that the composite has a periodic structure
with period $Y$ with dimension $\mathcal{O}\left( \ell ^{3}\right) $ where $%
\ell $ [m] is the microscopic characteristic length. The fundamental
assumption in the periodic homogenization theory \cite{blp, san1} is that
these scales are separated which in this case can be read as follows:%
\begin{equation*}
\varepsilon :=\frac{\ell }{L}\ll 1.
\end{equation*}%
We assume that the stiffness tensors $\mathbf{A}_{1}$ and $\mathbf{A}_{2}$,
permeabilities $\kappa _{1}$ and $\kappa _{2}$, thermal conductivities $\hat{%
\lambda}_{1}$ and $\hat{\lambda}_{2}$ are of the same order of magnitude.
Pressures $p_{1}$, $p_{2}$ and temperatures $\theta _{1}$, $\theta _{2}$ are
also considered of the same order. More precisely, we assume that
\begin{eqnarray*}
\left\vert \mathbf{A}_{1}\right\vert &=&\left\vert \mathbf{A}_{2}\right\vert
=\mathcal{O}\left( \varepsilon ^{0}\right) ,\ \left\vert \kappa
_{1}\right\vert =\left\vert \kappa _{2}\right\vert =\mathcal{O}\left(
\varepsilon ^{0}\right) ,\ \left\vert \hat{\lambda}_{1}\right\vert
=\left\vert \hat{\lambda}_{2}\right\vert =\mathcal{O}\left( \varepsilon
^{0}\right) , \\
p_{1} &=&p_{2}=\mathcal{O}\left( \varepsilon ^{0}\right) ,\ \theta
_{1}=\theta _{2}=\mathcal{O}\left( \varepsilon ^{0}\right) .
\end{eqnarray*}%
Equations (\ref{e26e}) and (\ref{e26f}) give at the microscale two
dimensionless (Biot) numbers:%
\begin{equation*}
Bi_{1}=\frac{\left\vert \varsigma \left( p_{1}-p_{2}\right) \right\vert }{%
\left\vert \kappa _{1}\nabla p_{1}\right\vert }=\frac{\varsigma l}{\kappa
_{1}},\quad Bi_{2}=\frac{\left\vert \omega \left( \theta _{1}-\theta
_{2}\right) \right\vert }{\left\vert \hat{\lambda}_{1}\nabla \theta
_{1}\right\vert }=\frac{\omega l}{\hat{\lambda}_{1}}.
\end{equation*}%
Since the area of $\Sigma ^{\varepsilon }\mathcal{\ }$is of order $%
\varepsilon ^{-1}$, a convenient scaling of those Biot numbers is ${\large B}%
i_{1}={\large B}i_{2}=\mathcal{O}\left( \varepsilon \right) $, see \cite%
{adhs, neus}.

\begin{remark}
\label{rem1}Note that other assumptions could be studied as well. One also
could consider the following case: $\left\vert \mathbf{A}_{1}\right\vert
=\left\vert \mathbf{A}_{2}\right\vert =\left\vert \kappa _{1}\right\vert
=\left\vert \hat{\lambda}_{1}\right\vert =\mathcal{O}\left( \varepsilon
^{0}\right) $,\ $p_{1}=p_{2}=\ \theta _{1}=\theta _{2}=\mathcal{O}\left(
\varepsilon ^{0}\right) $ and $\ \left\vert \kappa _{2}\right\vert
=\left\vert \hat{\lambda}_{2}\right\vert =\mathcal{O}\left( \varepsilon
^{2}\right) $.\ See for instance \cite{ain2, ain3} and for more general
situations we refer the reader to \cite{pb} .
\end{remark}

\subsection{Problem statement}

In this subsection a micro-model for a thermoporoelastic medium with
two-components and with interfacial hydraulic/thermal exchange barrier is
presented. As before, we consider $\Omega $\ a bounded domain in $\mathbb{R}%
^{3}\ $with a smooth boundary $\Gamma $. The region $\Omega $ represents a
part of a medium made of two constituents: the matrix and the inclusions,
separated by a thin and periodic layer so that the hydraulic/thermal flux
are proportional to the jump of the pressure/temperature field. To describe
the periodicity of the medium, we consider $Y:=]0,1[^{3}$ as the generic
cell of periodicity divided as $Y:=Y_{1}\cup Y_{2}\cup \Sigma $ where $%
Y_{1}, $\ $Y_{2}$\ are two connected, open and disjoint subsets of $Y$ and $%
\Sigma :=\partial Y_{1}\cap \partial Y_{2}$\ is a smooth surface that
separates them. We assume that $\bar{Y}_{2}\subset Y$. Let $\chi _{m}$
denote the $Y$-periodic characteristic function of $Y_{m}$ ($m=1,2$). Let $%
\varepsilon >0$ be a sufficiently small parameter and set%
\begin{equation*}
\Omega _{m}^{\varepsilon }:=\{x\in \Omega :\chi _{m}(\frac{x}{\varepsilon }%
)=1\},\ \ \ \ \Sigma ^{\varepsilon }:=\overline{\Omega _{1}^{\varepsilon }}%
\cap \overline{\Omega _{2}^{\varepsilon }}.
\end{equation*}%
We assume that $\bar{\Omega}_{2}^{\varepsilon }\subset \Omega $. The
space-time regions are denoted by%
\begin{eqnarray*}
&&Q:=\left( 0,T\right) \times \Omega ,\ \ \Gamma _{T}:=\left( 0,T\right)
\times \Gamma ,\ \Sigma _{T}:=\left( 0,T\right) \times \Sigma ,\  \\
&&\ Q_{m}^{\varepsilon }:=\ \left( 0,T\right) \times \Omega
_{m}^{\varepsilon },\ \ \Sigma _{T}^{\varepsilon }:=\left( 0,T\right) \times
\Sigma ^{\varepsilon }.
\end{eqnarray*}%
In view of (\ref{e25}), the thermoporoelasticity system is given in each
phase $Q_{m}^{\varepsilon }$ by%
\begin{equation}
-\mathrm{div}\sigma _{m}^{\varepsilon }=\mathbf{f}_{m},\mathbf{\ \ \ }\sigma
_{m}^{\varepsilon }=\left( \sigma _{m,ij}\right) _{ij},  \label{e29}
\end{equation}%
\begin{equation}
\partial _{t}\left( \phi _{m}p_{m}^{\varepsilon }+\beta _{m}\mathrm{div}%
\mathbf{u}_{m}^{\varepsilon }+\alpha _{m}\theta _{m}^{\varepsilon }\right)
-\kappa _{m}\Delta p_{m}^{\varepsilon }=g_{m},  \label{e30}
\end{equation}%
\begin{equation}
\partial _{t}\left( c_{m}\theta _{m}^{\varepsilon }+\gamma _{m}\mathrm{div}%
\mathbf{u}_{m}^{\varepsilon }+\alpha _{m}p_{m}^{\varepsilon }\right) -\hat{%
\lambda}_{m}\Delta \theta _{m}^{\varepsilon }=h_{m}  \label{e31}
\end{equation}%
where \
\begin{equation}
\sigma _{m,ij}^{\varepsilon }\left( \mathbf{u}^{\varepsilon }\right)
=a_{m,ijkl}e_{kl}\left( \mathbf{u}_{m}^{\varepsilon }\right) -\left( \beta
_{m}p_{m}^{\varepsilon }+\gamma _{m}\theta _{m}^{\varepsilon }\right) \delta
_{ij}  \label{art7-53}
\end{equation}%
with$\mathbf{\ }\left( a_{m,ijkl}\right) _{1\leq i,j,k,l\leq 3}$ the
elasticity tensor stiffness satisfying the Hooke's law for isotropic
materials:%
\begin{equation*}
a_{m,ijkl}=\lambda _{m}\delta _{ij}\delta _{kl}+\mu _{m}\left( \delta
_{ik}\delta _{jl}+\delta _{il}\delta _{jk}\right) ,\ m=1,2,\ 1\leq
i,j,k,l\leq 3.
\end{equation*}%
As in (\ref{e26a})-(\ref{e26f}), the transmission conditions on the
interface $\Sigma _{T}^{\varepsilon }$ are as follows:
\begin{subequations}
\label{e32}
\begin{align}
& \mathbf{u}_{1}^{\varepsilon }=\mathbf{u}_{2}^{\varepsilon },  \label{e32a}
\\
& \sigma _{1}^{\varepsilon }\cdot \mathbf{n}^{\varepsilon }=\sigma
_{2}^{\varepsilon }\cdot \mathbf{n}^{\varepsilon },  \label{e32b} \\
& \kappa _{1}\nabla p_{1}^{\varepsilon }\cdot \mathbf{n}^{\varepsilon
}=\kappa _{2}\nabla p_{2}^{\varepsilon }\cdot \mathbf{n}^{\varepsilon },
\label{e32c} \\
& \hat{\lambda}_{1}\nabla \theta _{1}^{\varepsilon }\cdot \mathbf{n}%
^{\varepsilon }=\hat{\lambda}_{2}\nabla \theta _{2}^{\varepsilon }\cdot
\mathbf{n}^{\varepsilon },  \label{e32d} \\
& \kappa _{1}\nabla p_{1}^{\varepsilon }\cdot \mathbf{n}^{\varepsilon
}=-\varsigma ^{\varepsilon }\left( p_{1}^{\varepsilon }-p_{2}^{\varepsilon
}\right) ,  \label{e32e} \\
& \hat{\lambda}_{1}\nabla \theta _{1}^{\varepsilon }\cdot \mathbf{n}%
^{\varepsilon }=-\omega ^{\varepsilon }\left( \theta _{1}^{\varepsilon
}-\theta _{2}^{\varepsilon }\right)  \label{e32f}
\end{align}%
where $\mathbf{n}^{\varepsilon }$ is the unit normal of $\Sigma
^{\varepsilon }$ pointing outwards of $\Omega _{2}^{\varepsilon }$.
Furthermore, boundary conditions (\ref{e27}) on $\Gamma _{T}$ read as:
\end{subequations}
\begin{equation}
\mathbf{u}_{1}^{\varepsilon }=0,\ \theta _{1}^{\varepsilon
}=p_{1}^{\varepsilon }=0.  \label{e33}
\end{equation}%
Observe that no boundary conditions on $\Gamma $ are required for the phase
2 since the inclusions $\bar{\Omega}_{2}^{\varepsilon }$ are strictly
embedded in $\Omega $. Finally, the initial conditions (\ref{e28}) give
\begin{equation}
\mathbf{u}_{m}^{\varepsilon }\left( x,0\right) =0,\ \theta _{m}^{\varepsilon
}\left( x,0\right) =p_{m}^{\varepsilon }\left( x,0\right) =0\text{ in }%
\Omega _{m}^{\varepsilon }.  \label{e34}
\end{equation}%
We assume that the elastic modulii $a_{m,ijkl}$, the Biot-Willis parameter $%
\beta _{m}$, the thermal dilation coefficients $\gamma _{m}$ and $\alpha
_{m} $, the compressibility $\phi _{m}$, the diffusivity coefficients $%
\kappa _{m},\hat{\lambda}_{m}$ and the\ heat capacity $c_{m}$ are positive
constants. We also assume that the body force $\mathbf{f}_{m}$ is in $%
L^{2}\left( \Omega \right) ^{3}$, the sink source $g_{m}$ and the heat
source $h_{m}$ are in $L^{2}\left( \Omega \right) $. Furthermore, the
interface hydraulic permeability and the interface thermal conductance are
such that $\varsigma ^{\varepsilon }\left( x\right) =\varepsilon \varsigma
\left( \frac{x}{\varepsilon }\right) $ and $\omega ^{\varepsilon }\left(
x\right) =\varepsilon \omega \left( \frac{x}{\varepsilon }\right) $ (see
Sec. \ref{sec1}) where $\varsigma $ and $\omega $ are continuous on $\mathbb{%
R}^{3}$, $Y-$periodic and bounded from below: $\exists C>0$ such that for
all $y\in \mathbb{R}^{3}$
\begin{equation*}
\varsigma \left( y\right) \geq C,\text{\hspace{0.5cm}}\omega \left( y\right)
\geq C\text{.}
\end{equation*}%
In what follows, $C$ will denote a positive constant independent of $%
\varepsilon $.

\section{Statement of the main results\label{vfmr}}

We first introduce some notations: if $E$\ is a Banach space then for $%
p=2,\infty $, $L_{T}^{p}(E)$ denotes the Bochner space $L^{p}(0,T;E)$ of
(class of ) functions $u:t\longmapsto u\left( t\right) $\ defined a.e. on $%
\left( 0,T\right) $\ with values in $E$\ such that $\left\Vert u\left(
t\right) \right\Vert _{L_{T}^{p}(E)}^{p}:=\int_{0}^{T}\left\Vert u\left(
t\right) \right\Vert _{E}^{p}{}\mathrm{d}t$\ is finite (${}\mathrm{d}t$\
denotes the Lebesgue measure on the interval $\left( 0,T\right) $). Let $%
L_{\#}^{2}(Y)$ (resp. $L_{\#}^{2}(Y_{m})$) be the space of (class of)
functions belonging to $L_{\mathrm{loc}}^{2}(\mathbb{R}^{3})$ (resp. $L_{%
\mathrm{loc}}^{2}(Z_{m})$) which are $Y$-periodic, where $Z_{m}=\cup _{\vec{k%
}\in \mathbb{Z}^{3}}\left( Y_{m}+\vec{k}\right) $.$\ $Let $H_{\#}^{1}(Y)$
(resp. $H_{\#}^{1}(Y_{m})$) to be the space of those functions together with
their derivatives belonging to $L_{\#}^{2}(Y)$ (resp. $L_{\#}^{2}(Y_{m})$)
having the same trace on the opposite faces of $\partial Y$ (resp. $\partial
Y_{m}\cap \partial Y$). Let
\begin{equation*}
\left.
\begin{array}{l}
\mathbf{V}\mathbf{:}=H_{0}^{1}\left( \Omega \right) ^{3},\quad
H^{\varepsilon }:=L^{2}\left( \Omega _{1}^{\varepsilon }\right) \times
L^{2}\left( \Omega _{2}^{\varepsilon }\right) , \\
\  \\
H_{\Gamma }^{1}\left( \Omega _{1}^{\varepsilon }\right) :=\left\{ q\in
H^{1}\left( \Omega _{1}^{\varepsilon }\right) :q_{|\Gamma }=0\right\} , \\
\  \\
V^{\varepsilon }:=V_{1}^{\varepsilon }\times V_{2}^{\varepsilon }=H_{\Gamma
}^{1}\left( \Omega _{1}^{\varepsilon }\right) \times H^{1}\left( \Omega
_{2}^{\varepsilon }\right) .%
\end{array}%
\right.
\end{equation*}%
The space $V^{\varepsilon }$ is equipped with the inner product:
\begin{equation*}
\begin{array}{l}
\left( q,\psi \right) _{V^{\varepsilon }}:=\int_{\Omega _{1}^{\varepsilon
}}\nabla q_{1}\nabla \psi _{1}\hspace{0.03cm}\mathrm{d}x+\int_{\Omega
_{2}^{\varepsilon }}\nabla q_{2}\nabla \psi _{2}\hspace{0.03cm}\mathrm{d}x+
\\
\\
\varepsilon \int_{\Sigma ^{\varepsilon }}\left( q_{1}-q_{2}\right) \left(
\psi _{1}-\psi _{2}\right) \hspace{0.03cm}\mathrm{d}s^{\varepsilon }\left(
x\right) ,\ \ q=\left( q_{1},q_{2}\right) ,\ \psi =\left( \psi _{1},\psi
_{2}\right) \in V^{\varepsilon }%
\end{array}%
\end{equation*}%
where $\mathrm{d}x$ and $\mathrm{d}s^{\varepsilon }\left( x\right) $ stands
for the Lebesgue measure in $\mathbb{R}^{3}$ and the surfacic measure on $%
\Sigma ^{\varepsilon }$, respectively. Let us denote for a.e. $t\in \left(
0,T\right) $%
\begin{eqnarray*}
\mathbf{u}^{\varepsilon }(t,x) &=&\left\{
\begin{array}{c}
\mathbf{u}_{1}^{\varepsilon }(t,x),\ \ x\in \Omega _{1}^{\varepsilon }, \\
\  \\
\mathbf{u}_{2}^{\varepsilon }(t,x),\ \ x\in \Omega _{2}^{\varepsilon },%
\end{array}%
\right. \\
&&\  \\
p^{\varepsilon }(t,x) &=&\left( p_{1}^{\varepsilon }(t,x),p_{2}^{\varepsilon
}(t,x)\right) ,\ \theta ^{\varepsilon }(t,x)=\left( \theta _{1}^{\varepsilon
}(t,x),\theta _{2}^{\varepsilon }(t,x)\right)
\end{eqnarray*}%
and let us define
\begin{equation}
\ \left.
\begin{array}{l}
\mathbf{A}^{\varepsilon }(x):=\chi _{1}(\frac{x}{\varepsilon })\mathbf{A}%
_{1}+\chi _{2}(\frac{x}{\varepsilon })\mathbf{A}_{2}, \\
\  \\
\mathbf{f}^{\varepsilon }(x):=\chi _{1}(\frac{x}{\varepsilon })\mathbf{f}%
_{1}\left( x\right) +\chi _{2}(\frac{x}{\varepsilon })\mathbf{f}_{2}\left(
x\right) , \\
\  \\
g^{\varepsilon }(x):=\chi _{1}(\frac{x}{\varepsilon })g_{1}\left( x\right)
+\chi _{2}(\frac{x}{\varepsilon })g_{2}\left( x\right) , \\
\  \\
h^{\varepsilon }(x):=\chi _{1}(\frac{x}{\varepsilon })h_{1}\left( x\right)
+\chi _{2}(\frac{x}{\varepsilon })h_{2}\left( x\right)%
\end{array}%
\right.  \label{e35}
\end{equation}%
where $\mathbf{A}_{m}:=\left( a_{m,ijkl}\right) _{1\leq i,j,k,l\leq 3}$. Now
we are in position to give the weak formulation.

\begin{definition}
A weak solution of the micro-model (\ref{e29})-(\ref{e34}) is a triple $%
\left( \mathbf{u}^{\varepsilon },p^{\varepsilon },\theta ^{\varepsilon
}\right) \in L_{T}^{\infty }(\mathbf{V})\times L_{T}^{2}(V^{\varepsilon
})^{2}$ such that $p^{\varepsilon },\theta ^{\varepsilon }\in L_{T}^{\infty
}(H^{\varepsilon })$ and for $m=1,2$
\begin{eqnarray*}
\partial _{t}\left( \phi _{m}p_{m}^{\varepsilon }+\beta _{m}\mathrm{div}%
\mathbf{u}_{m}^{\varepsilon }+\alpha _{m}\theta _{m}^{\varepsilon }\right)
&\in &L_{T}^{2}(V_{m}^{\varepsilon }{}^{\ast }), \\
\partial _{t}\left( c_{m}\theta _{m}^{\varepsilon }+\gamma _{m}\mathrm{div}%
\mathbf{u}_{m}^{\varepsilon }+\alpha _{m}p_{m}^{\varepsilon }\right) &\in
&L_{T}^{2}(V_{m}^{\varepsilon }{}^{\ast })
\end{eqnarray*}%
and for all $\mathbf{v}\in \mathbf{V}$, $(q_{1},q_{2})\in V^{\varepsilon }$,
we have the three following coupled systems: for$\ $a.e. $t\in \left(
0,T\right) ,$
\begin{equation}
\left\{
\begin{array}{l}
\int_{\Omega }\mathbf{A}^{\varepsilon }e(\mathbf{u}^{\varepsilon })e(\mathbf{%
v})\hspace{0.03cm}\mathrm{d}x+\int_{\Omega _{1}^{\varepsilon }}\left( \beta
_{1}\nabla p_{1}^{\varepsilon }+\gamma _{1}\nabla \theta _{1}^{\varepsilon
}\right) \mathbf{v}\hspace{0.03cm}\mathrm{d}x \\
\\
+\int_{\Omega _{2}^{\varepsilon }}\left( \beta _{2}\nabla p_{2}^{\varepsilon
}+\gamma _{2}\nabla \theta _{2}^{\varepsilon }\right) \mathbf{v}\hspace{%
0.03cm}\mathrm{d}x=\int_{\Omega }\mathbf{f}^{\varepsilon }\mathbf{v}\hspace{%
0.03cm}\mathrm{d}x,%
\end{array}%
\right.  \label{e36}
\end{equation}%
\begin{equation}
\left\{
\begin{array}{l}
\langle \partial _{t}(\phi _{1}p_{1}^{\varepsilon }+\beta _{1}\mathrm{div}%
\mathbf{u}^{\varepsilon }+\alpha _{1}\theta _{1}^{\varepsilon
}),q_{1}\rangle _{V_{1}^{\varepsilon }{}^{\ast },V_{1}^{\varepsilon }} \\
\\
+\langle \partial _{t}(\phi _{2}p_{2}^{\varepsilon }+\beta _{2}\mathrm{div}%
\mathbf{u}^{\varepsilon }+\alpha _{2}\theta _{2}^{\varepsilon
}),q_{2}\rangle _{V_{2}^{\varepsilon }{}^{\ast },V_{2}^{\varepsilon }} \\
\\
+\int_{\Omega _{1}^{\varepsilon }}\kappa _{1}\nabla p_{1}^{\varepsilon
}\nabla q_{1}\hspace{0.03cm}\mathrm{d}x+\int_{\Omega _{2}^{\varepsilon
}}\kappa _{2}\nabla p_{2}^{\varepsilon }\nabla q_{2}\hspace{0.03cm}\mathrm{d}%
x \\
\\
+\int_{\Sigma ^{\varepsilon }}\varsigma ^{\varepsilon }(p_{1}^{\varepsilon
}-p_{2}^{\varepsilon })(q_{1}-q_{2})\hspace{0.03cm}\mathrm{d}s^{\varepsilon
}(x)=\int_{\Omega }g^{\varepsilon }q^{\varepsilon }\hspace{0.03cm}\mathrm{d}x%
\end{array}%
\right.  \label{e37}
\end{equation}%
and%
\begin{equation}
\left\{
\begin{array}{l}
\langle \partial _{t}(c_{1}\theta _{1}^{\varepsilon }+\gamma _{1}\mathrm{div}%
\mathbf{u}^{\varepsilon }+\alpha _{1}p_{1}^{\varepsilon }),q_{1}\rangle
_{V_{1}^{\varepsilon }{}^{\ast },V_{1}^{\varepsilon }} \\
\\
\langle \partial _{t}(c_{2}\theta _{2}^{\varepsilon }+\gamma _{2}\mathrm{div}%
\mathbf{u}^{\varepsilon }+\alpha _{2}p_{2}^{\varepsilon }),q_{2}\rangle
_{V_{2}^{\varepsilon }{}^{\ast },V_{2}^{\varepsilon }} \\
\\
+\int_{\Omega _{1}^{\varepsilon }}\hat{\lambda}_{1}\nabla \theta
_{1}^{\varepsilon }\nabla q_{1}\hspace{0.03cm}\mathrm{d}x+\int_{\Omega
_{2}^{\varepsilon }}\hat{\lambda}_{2}\nabla \theta _{2}^{\varepsilon }\nabla
q_{2}\hspace{0.03cm}\mathrm{d}x \\
\\
+\int_{\Sigma ^{\varepsilon }}\omega ^{\varepsilon }(\theta
_{1}^{\varepsilon }-\theta _{2}^{\varepsilon })(q_{1}-q_{2})\hspace{0.03cm}%
\mathrm{d}s^{\varepsilon }(x)=\int_{\Omega ^{\varepsilon }}h^{\varepsilon
}q^{\varepsilon }\hspace{0.03cm}\mathrm{d}x%
\end{array}%
\right.  \label{e38}
\end{equation}%
with the initial conditions:
\begin{equation}
\left\{
\begin{array}{l}
\mathbf{u}^{\varepsilon }(0,\cdot )=\mathbf{0}, \\
p_{1}^{\varepsilon }(0,x)=\theta _{1}^{\varepsilon }(0,x)=0,\ \ x\in \Omega
_{1}^{\varepsilon }, \\
p_{2}^{\varepsilon }(0,x)=\theta _{2}^{\varepsilon }(0,x)=0,\ \ x\in \Omega
_{2}^{\varepsilon }%
\end{array}%
\right.  \label{e39}
\end{equation}%
where we have denoted%
\begin{equation*}
q^{\varepsilon }\left( x\right) =\left\{
\begin{array}{c}
q_{1}\left( x\right) ,\ \ x\in \Omega _{1}^{\varepsilon }, \\
q_{2}\left( x\right) ,\ \ x\in \Omega _{2}^{\varepsilon }.%
\end{array}%
\right.
\end{equation*}
\end{definition}

Existence and uniqueness results \ for the system (\ref{e36})-(\ref{e39})
can be performed by using the Galerkin technique or the semi-group method.
For more details, we refer the reader for instance to Showalter and Momken
\cite{sm}. Hence we give without proof the following result.

\begin{theorem}
There exists a unique $($\textbf{$u$}$^{\varepsilon },p^{\varepsilon
},\theta ^{\varepsilon })\in L_{T}^{\infty }(\mathbf{V})\times
L_{T}^{2}(V^{\varepsilon })^{2}$ with $(p^{\varepsilon },\theta
^{\varepsilon })\in L_{T}^{\infty }(H^{\varepsilon })^{2}$ solution of the
weak system (\ref{e29})-(\ref{e34}) such that
\begin{equation}
\left.
\begin{array}{l}
\Vert \mathbf{u}^{\varepsilon }\Vert _{L_{T}^{\infty }(\mathbf{V})}\leq C,
\\
\Vert p^{\varepsilon }\Vert _{L_{T}^{2}(V^{\varepsilon })}+\Vert
p^{\varepsilon }\Vert _{L_{T}^{\infty }(H^{\varepsilon })}\leq C, \\
\Vert \theta ^{\varepsilon }\Vert _{L_{T}^{2}(V^{\varepsilon })}+\Vert
\theta ^{\varepsilon }\Vert _{L_{T}^{\infty }(H^{\varepsilon })}\leq C.%
\end{array}%
\right.  \label{e40}
\end{equation}
\end{theorem}

The key ingredient of our convergence results are the uniform estimates (\ref%
{e40}) and the use of the two-scale convergence technique, see Sect. \ref%
{sec2} below. In order to give our main result, we introduce some notations
related to the local scale models: let $\mathbf{d}^{kl}=(y_{k}\delta
_{il})_{1\leq i\leq 3}$ and $\mathbf{w}^{ij}\in (H_{\#}^{1}(Y)/\mathbb{R}%
)^{3}$ denote the solution to the microscopic system:%
\begin{equation}
\left\{
\begin{array}{ll}
-\mathrm{div}_{y}(\mathbf{A}(e_{y}\left( \mathbf{w}^{ij}+\mathbf{d}%
^{ij}\right) ))=0 & \text{ a.e. in }Y, \\
\left[ \mathbf{w}^{ij}\right] =0 & \text{ a.e. on }\Sigma , \\
y\longmapsto \mathbf{w}^{ij}(x,y) & \text{ is }Y-\text{periodic}%
\end{array}%
\right.  \label{e41}
\end{equation}%
where
\begin{equation*}
\mathbf{A}\left( y\right) :=\left\{
\begin{array}{c}
\mathbf{A}_{1},\ \ y\in Y_{1}, \\
\mathbf{A}_{2},\ \ y\in Y_{2}.%
\end{array}%
\right.
\end{equation*}%
and $\left[ \mathbf{\cdot }\right] $ denotes the jump on $\Sigma $. Let us
define the "micro-pressure" $\pi _{m}^{i}\in H_{\#}^{1}(Y_{m})/\mathbb{R}$ ($%
i=1,2,3$) to be the solution of
\begin{equation}
\left\{
\begin{array}{ll}
-\mathrm{div}\left( \kappa _{m}\left( e^{i}+\nabla _{y}\pi _{m}^{i}\left(
y\right) \right) \right) =0 & \text{ in }\Omega \times Y_{m}, \\
\left[ \kappa _{m}\left( e^{i}+\nabla _{y}\pi _{m}^{i}\left( y\right)
\right) \right] \cdot \nu \left( y\right) =0 & \text{ on }\Omega \times
\Sigma , \\
y\longmapsto \pi _{m}^{i}\left( y\right) & \ \text{is }Y-\text{periodic.}%
\end{array}%
\right.  \label{e42}
\end{equation}%
Here $e^{i}$ is the $i^{\text{th}}$ vector of the standard basis of $\mathbb{%
R}^{3}$. Likewise, let the "micro-temperature" $\vartheta _{m}^{i}\in
H^{1}(Y_{m})/\mathbb{R}$ to be the solution of
\begin{equation}
\left\{
\begin{array}{ll}
-\mathrm{div}\left( \hat{\lambda}_{m}\left( y\right) \left( e^{i}+\nabla
_{y}\vartheta _{m}^{i}\right) \right) =0 & \text{in }\Omega \times Y_{m}, \\
\left[ \hat{\lambda}_{m}\left( y\right) \left( e^{i}+\nabla _{y}\vartheta
_{m}^{i}\right) \right] \cdot \nu =0 & \text{on }\Omega \times \Sigma , \\
y\longmapsto \vartheta _{m}^{i} & \text{is }Y-\text{periodic.}%
\end{array}%
\right.  \label{e43}
\end{equation}

\begin{remark}
It is worthwhile noticing that these three boundary value problems (\ref{e41}%
), (\ref{e42}) and (\ref{e43}) are well-posed in the sense that each problem
admits a unique weak solution (see for instance A. Bensoussan \& al. \cite%
{blp}).
\end{remark}

Let us denote
\begin{equation}
\mathcal{A}_{ijkl}:=\int_{Y}\mathbf{A}e_{y}\left( \mathbf{w}^{ij}+\mathbf{d}%
^{ij}\right) e_{y}\left( \mathbf{w}^{kl}+\mathbf{d}^{kl}\right) \hspace{%
0.03cm}\mathrm{d}x,\ \sigma _{ij}^{0}\left( \mathbf{u}\right) :=\mathcal{A}%
_{ijkl}e_{kh}(\mathbf{u}).  \label{e44}
\end{equation}%
In other words
\begin{equation}
\mathcal{A}_{ijkh}=\sum_{r,s=1}^{3}\int_{Y}a_{ijrs}(y)(\delta _{ir}\delta
_{js}+e_{rs,y}(\mathbf{w}^{kh})(y))\hspace{0.03cm}\mathrm{d}y.  \label{e45}
\end{equation}

Put
\begin{eqnarray*}
B_{m} &=&\left( B_{m,ij}\right) _{1\leq i,j\leq 3},D_{m}=\left(
D_{m,ij}\right) _{1\leq i,j\leq 3},\ \mathcal{C}_{m}=\left( \mathcal{C}%
_{m,ij}\right) _{1\leq i,j\leq 3}, \\
\mathcal{K}_{m} &=&\left( \mathcal{K}_{m,ij}\right) _{1\leq i,j\leq 3},\
\mathcal{L}_{m}=\left( \mathcal{L}_{m,ij}\right) _{1\leq i,j\leq 3}
\end{eqnarray*}%
where
\begin{eqnarray}
B_{m,ij} &:&=\beta _{m}\int_{Y}\chi _{m}\left( y\right) \left( \delta _{ij}+%
\frac{\partial \pi _{m}^{i}}{\partial y_{j}}(y)\right) \hspace{0.03cm}%
\mathrm{d}y,  \label{e46} \\
D_{m,ij} &:&=\gamma _{m}\int_{Y}\chi _{m}\left( y\right) \left( \delta _{ij}+%
\frac{\partial \vartheta _{m}^{i}}{\partial y_{j}}(y)\right) \hspace{0.03cm}%
\mathrm{d}y,  \label{e47} \\
\mathcal{C}_{m,ij} &:&=\int_{Y_{m}}\left( \delta _{j}^{i}+\mathrm{div}\left(
\mathbf{w}^{ij}\right) \right) \hspace{0.03cm}\mathrm{d}y  \label{e48}
\end{eqnarray}%
and
\begin{equation}
\mathcal{K}_{m,ij}:=\int_{Y_{m}}\kappa _{m}\left( \delta _{i}^{j}+\frac{%
\partial \pi _{mj}}{\partial y_{i}}\right) \hspace{0.03cm}\mathrm{d}y,\
\mathcal{L}_{m,ij}:=\int_{Y_{m}}\hat{\lambda}_{m}\left( \delta _{i}^{j}+%
\frac{\partial \vartheta _{mj}}{\partial y_{i}}\right) \hspace{0.03cm}%
\mathrm{d}y.  \label{e49}
\end{equation}

Set
\begin{equation}
\left\{
\begin{array}{ccccc}
\mathbf{f}^{\ast }:=\left\vert Y_{1}\right\vert \mathbf{f}_{1}\mathbf{+}%
\left\vert Y_{2}\right\vert \mathbf{f}_{2}, &  & g_{m}^{\ast }:=\left\vert
Y_{m}\right\vert g_{m},\  &  & h_{m}^{\ast }:=\left\vert Y_{m}\right\vert
h_{m}, \\
&  &  &  &  \\
c_{m}^{\ast }:=\left\vert Y_{m}\right\vert c_{m}, &  & \varphi _{m}^{\ast
}:=\left\vert Y_{m}\right\vert \phi _{m}, &  & \gamma _{m}^{\ast
}:=\left\vert Y_{m}\right\vert \gamma _{m}, \\
&  &  &  &  \\
\omega ^{\ast }:=\int_{\Sigma }\omega \left( y\right) \hspace{0.03cm}\mathrm{%
d}s\left( y\right) , &  & \zeta ^{\ast }:=\int_{\Sigma }\varsigma \left(
y\right) \hspace{0.03cm}\mathrm{d}s\left( y\right) . &  & \alpha _{m}^{\ast
}:=\left\vert Y_{m}\right\vert \alpha _{m}%
\end{array}%
\right.  \label{e50}
\end{equation}%
where $\left\vert Y_{m}\right\vert $ denotes the volume of $Y_{m}$ and $%
\mathrm{d}s\left( y\right) $ the surfacic measure of $\Sigma $.

Now we state the main result of the paper.

\begin{theorem}
\label{t1}Let $($\textbf{$u$}$^{\varepsilon },p^{\varepsilon },\theta
^{\varepsilon })\in L_{T}^{\infty }(\mathbf{V})\times
L_{T}^{2}(V^{\varepsilon })^{2}$ be the weak solution of (\ref{e29})-(\ref%
{e34}). Then, up to a subsequence, there exist $\mathbf{u}\in L_{T}^{\infty
}(\mathbf{V})$, $p_{1},\theta _{1}\in L_{T}^{\infty }(H_{0}^{1}(\Omega ))$
and $p_{2},\theta _{2}\in L_{T}^{\infty }(H^{1}(\Omega ))$ such that the
following weak limits holds: for a.e. $t\in (0,T)$
\begin{eqnarray*}
\mathbf{u}^{\varepsilon }(t,\cdot ) &\rightharpoonup &\mathbf{u}(t,\cdot )%
\mathbf{,}\ \text{weakly in }L_{T}^{\infty }(\mathbf{V}), \\
\chi _{m}^{\varepsilon }p_{m}^{\varepsilon }(t,\cdot ) &\rightharpoonup
&\chi _{m}p_{m\,}(t,\cdot ),\ \theta _{m}^{\varepsilon }(t,\cdot )\overset{%
2-s}{\rightharpoonup }\chi _{m}\theta _{m\,}(t,\cdot ),\ \text{weakly in }%
L_{T}^{\infty }(H^{1}(\Omega )).
\end{eqnarray*}%
Furthermore, the weak limit $\left( \mathbf{u},p_{m},\theta _{m}\right) $\
is solution of the following homogenized system:%
\begin{equation}
\left\{
\begin{array}{l}
-\mathrm{div}\left( \sigma ^{0}\left( \mathbf{u}\right) \right) +B_{1}\nabla
p_{1}+B_{2}\nabla p_{2}+D_{1}\nabla \theta _{1}+D_{2}\nabla \theta _{2}=%
\mathbf{f}\text{, a.e in }Q\text{,} \\
\\
\partial _{t}\left( \phi _{1}^{\ast }p_{1}+\beta _{1}\mathcal{C}_{1}:e\left(
\mathbf{u}\right) +\alpha _{1}^{\ast }\theta _{1}\right) -\mathrm{div}\left(
\mathcal{K}_{1}\nabla p_{1}\right) +\zeta ^{\ast }\left( p_{1}-p_{2}\right)
=g_{1}^{\ast }\text{, a.e. in }Q\text{,} \\
\\
\partial _{t}\left( \phi _{2}^{\ast }p_{2}+\beta _{2}\mathcal{C}_{2}:e\left(
\mathbf{u}\right) +\alpha _{2}^{\ast }\theta _{2}\right) -\mathrm{div}\left(
\mathcal{K}_{2}\nabla p_{2}\right) +\zeta ^{\ast }\left( p_{2}-p_{1}\right)
=g_{2}^{\ast }\text{, a.e.in }Q\text{,} \\
\\
\partial _{t}\left\{ c_{1}^{\ast }\theta _{1}+\gamma _{1}\mathcal{C}%
_{1}:e\left( \mathbf{u}\right) +\alpha _{1}^{\ast }p_{1}\right\} -\mathrm{div%
}\left( \mathcal{L}_{1}\nabla \theta _{1}\right) +\omega ^{\ast }\left(
\theta _{1}-\theta _{2}\right) =h_{1}^{\ast }\text{, a.e.in }Q\text{,} \\
\\
\partial _{t}\left\{ c_{2}^{\ast }\theta _{2}+\gamma _{2}\mathcal{C}%
_{2}:e\left( \mathbf{u}\right) +\alpha _{2}^{\ast }p_{2}\right\} -\mathrm{div%
}\left( \mathcal{L}_{2}\nabla \theta _{2}\right) +\omega ^{\ast }\left(
\theta _{2}-\theta _{1}\right) =h_{2}^{\ast }\text{, a.e.in }Q\text{,} \\
\\
\mathbf{u}=0,\ \theta _{1}=p_{1}=0\quad \text{a.e. on }\Gamma _{T}, \\
\\
\mathcal{L}_{2}\nabla \theta _{2}\cdot \nu =\mathcal{K}_{2}\nabla p_{2}\cdot
\nu =0,\quad \text{a.e. on }\Gamma _{T}, \\
\\
\mathbf{u}(0,x)=\mathbf{0},\quad \theta _{m}\left( 0,x\right)
=p_{m}(0,x)=0\quad \text{a.e. in }\Omega .%
\end{array}%
\right.  \label{e51}
\end{equation}
\end{theorem}

The next section is devoted to the proof of Theorems \ref{t1}.

\section{Derivation of the homogenized model\label{dhm}}

In this section, we shall first recall \ the two scale convergence
technique. Then we shall use the uniform estimates (\ref{e40}) to show that
the solution to the microscopic model (\ref{e29})-(\ref{e34}) converges in
the two-scale sense to the solution of the homogenized problem (\ref{e51}).

\subsection{Two scale convergence\label{sec2}}

The two-scale convergence method was first introduced by G. Nguetseng \cite%
{ngue} and later developed by G. Allaire \cite{all}. This technique is
intended to handle homogenization problems involving periodic
microstructures. Hereafter, we recall its definition and its main results.
For more details, we refer the reader to \cite{all, lnw, ngue}\textit{.}

We denote $\mathcal{C}_{\#}(Y)$ to be the space of all continuous functions
in $\mathbb{R}^{3}$ which are $Y$-periodic. Let $\mathcal{C}_{\#}^{\infty
}(Y)$ denote the subspace of $\mathcal{C}_{\#}(Y)$ of infinitely
differentiable functions.\textsl{\ }

\begin{lemma}
\label{l1}Let $q\in L^{2}(\Omega ;C_{\#}(Y))$. Then $q\left( x,x/\varepsilon
\right) \in L^{2}\left( \Omega \right) $ and
\begin{eqnarray*}
\int_{\Omega }\left\vert q(x,\frac{x}{\varepsilon })\right\vert ^{2}\mathrm{d%
}x &\leq &\int_{\Omega }\sup_{y\in Y}\left\vert q(x,y)\right\vert ^{2}%
\hspace{0.03cm}\mathrm{d}x, \\
\lim_{\varepsilon \rightarrow 0}\int_{\Omega }\left\vert q(x,\frac{x}{%
\varepsilon })\right\vert ^{2}\mathrm{d}x &=&\lim_{\varepsilon \rightarrow
0}\int_{\Omega \times Y}q(x,y)^{2}\hspace{0.03cm}\mathrm{d}x\mathrm{d}y.
\end{eqnarray*}
\end{lemma}

Such a function will be called in the sequel an admissible test function.

\begin{definition}
\label{d1} A sequence $(v^{\varepsilon })\ $in $L^{2}(\Omega )$ two-scale
converges to $v\in L^{2}(\Omega \times Y)$ and we write $v^{\varepsilon }%
\overset{2-s}{\rightharpoonup }v$ if, for any $q\in L^{2}(\Omega ;C_{\#}(Y))$%
,
\begin{equation*}
\lim_{\varepsilon \rightarrow 0}\int_{\Omega }v^{\varepsilon }(x)q(x,\frac{x%
}{\varepsilon })\hspace{0.03cm}\mathrm{d}x=\int_{\Omega \times Y}v(x,y)q(x,y)%
\hspace{0.03cm}\mathrm{d}x\mathrm{d}y.
\end{equation*}%
\
\end{definition}

\begin{theorem}
\label{t3} Let $(v^{\varepsilon })$ be a sequence of functions in $%
L^{2}(\Omega )$ which is uniformly bounded. Then, there exist $v\in
L^{2}(\Omega \times Y)$ and a subsequence of $(v^{\varepsilon })$ which
two-scale converges to $v$.
\end{theorem}

\begin{remark}
\label{r1}Thanks to Theorem \ref{t3}, it is easy to see that for all $q\in
L^{2}(\Omega ;C_{\#}(Y))$%
\begin{equation*}
\lim_{\varepsilon \rightarrow 0}\int_{\Omega _{m}^{\varepsilon }}\mathbf{f}%
_{m}(x)q(x,\frac{x}{\varepsilon })\hspace{0.03cm}\mathrm{d}x=\int_{\Omega
\times Y_{m}}\mathbf{f}_{m}(x)q(x,y)\hspace{0.03cm}\mathrm{d}x\mathrm{d}y
\end{equation*}%
since
\begin{equation*}
\int_{\Omega _{m}^{\varepsilon }}\mathbf{f}_{m}(x)q(x,\frac{x}{\varepsilon })%
\hspace{0.03cm}\mathrm{d}x=\int_{\Omega }\mathbf{f}_{m}(x)\chi _{m}\left(
\frac{x}{\varepsilon }\right) q(x,\frac{x}{\varepsilon })\hspace{0.03cm}%
\mathrm{d}x
\end{equation*}%
and $x\longmapsto \chi _{m}\left( \frac{x}{\varepsilon }\right) q(x,\frac{x}{%
\varepsilon })$ is an admissible test function.
\end{remark}

\begin{theorem}
\label{t4} Let $(v^{\varepsilon })$ be a uniformly bounded sequence in $%
H^{1}(\Omega )$ (resp. $H_{0}^{1}(\Omega )$). Then there exist $v\in
H^{1}(\Omega )$ (resp. $H_{0}^{1}(\Omega )$) and $\hat{v}\in L^{2}(\Omega
;H_{\#}^{1}(Y)/\mathbb{R})$ such that, up to a subsequence,
\begin{equation}
v^{\varepsilon }\overset{2-s}{\rightharpoonup }v;\quad \nabla v^{\varepsilon
}\overset{2-s}{\rightharpoonup }\nabla v+\nabla _{y}\hat{v}.  \label{t4e1}
\end{equation}
\end{theorem}

In (\ref{t4e1}) and in the sequel the subscript $y$ on a differential
operator as in $\nabla _{y}$ indicates that the differentiation acts only on
$y$.

\begin{theorem}
operators act only on those variables
\end{theorem}

We now extend the notion of two-scale convergence to periodic surfaces \cite%
{adhs, neus}:

\begin{definition}
\label{d2}A sequence $(w^{\varepsilon })$ in $L^{2}(\Sigma ^{\varepsilon })$
two-scale converges to $w_{0}(x,y)\in L^{2}\left( \Omega \times \Sigma
\right) $ if for any $q\in \mathcal{D}\left( \bar{\Omega};\mathcal{C}%
_{\#}^{\infty }(\Sigma )\right) $
\begin{equation*}
\lim_{\varepsilon \rightarrow 0}\varepsilon \int_{\Sigma ^{\varepsilon
}}w^{\varepsilon }(x)q(x,\frac{x}{\varepsilon })\hspace{0.03cm}\mathrm{d}%
s^{\varepsilon }\left( x\right) =\int_{\Omega \times \Sigma }w_{0}(x,y)q(x,y)%
\hspace{0.03cm}\mathrm{d}x\mathrm{d}s\left( y\right) .
\end{equation*}
\end{definition}

We state the following compactness result:

\begin{theorem}
\label{t5}Let $(w^{\varepsilon })$ be a sequence in $L^{2}(\Sigma
^{\varepsilon })$ such that
\begin{equation*}
\sqrt{\varepsilon }\int_{\Sigma ^{\varepsilon }}\left\vert w^{\varepsilon
}(x)\right\vert ^{2}\hspace{0.03cm}\mathrm{d}s^{\varepsilon }\left( x\right)
\leq C.
\end{equation*}%
Then, up to a subsequence, there exists $w_{0}(x,y)\in L^{2}\left( \Omega
\times \Sigma \right) $ such that $(w^{\varepsilon })$ two-scale converges
in the sense of Definition \ref{d2} to $w_{0}(x,y)\in L^{2}\left( \Omega
\times \Sigma \right) $.
\end{theorem}

\begin{corollary}
\label{c1}Let $v\left( y\right) \in L_{\#}^{2}\left( \Sigma \right) $. Then
for any $q\in H^{1}\left( \Omega \right) $%
\begin{equation*}
\lim_{\varepsilon \rightarrow 0}\int_{\Sigma ^{\varepsilon }}\varepsilon v(%
\frac{x}{\varepsilon })q(x)\hspace{0.03cm}\mathrm{d}s^{\varepsilon }\left(
x\right) =\int_{\Omega \times \Sigma }v(y)q(x)\hspace{0.03cm}\mathrm{d}x%
\mathrm{d}s\left( y\right) .
\end{equation*}
\end{corollary}

\begin{theorem}
\label{t6}Let $(w^{\varepsilon })$ be a sequence of functions in $%
H^{1}(\Omega )$ such that
\begin{equation*}
\left\Vert w^{\varepsilon }\right\Vert _{L^{2}\left( \Omega \right)
}+\varepsilon \left\Vert \nabla w^{\varepsilon }\right\Vert _{L^{2}\left(
\Omega \right) ^{N}}\leq C\text{.}
\end{equation*}%
Then, there exist a subsequence of $\left( w^{\varepsilon }\right) $, still
denoted by $\left( w^{\varepsilon }\right) $, and $w_{0}(x,y)\in L^{2}\left(
\Omega ;H_{\#}^{1}(Y)\right) $ such that $w^{\varepsilon }\overset{2-s}{%
\rightharpoonup }w_{0}$ and $\varepsilon \nabla w^{\varepsilon }\overset{2-s}%
{\rightharpoonup }\nabla _{y}w_{0}$ and for every $q\in \mathcal{D}\left(
\overline{\Omega };\mathcal{C}_{\#}^{\infty }(Y)\right) $ we have
\begin{equation*}
\lim_{\varepsilon \rightarrow 0}\varepsilon \int_{\Sigma ^{\varepsilon
}}w^{\varepsilon }\left( x\right) q^{\varepsilon }\left( x\right) \hspace{%
0.03cm}\mathrm{d}s^{\varepsilon }\left( x\right) =\int_{\Omega \times \Sigma
}w_{0}\left( x,y\right) q\left( x,y\right) \hspace{0.03cm}\mathrm{d}x\mathrm{%
d}s\left( y\right) .
\end{equation*}
\end{theorem}

\begin{remark}
Notice that two-scale convergence can also handle problems involving a
parameter without affecting the results stated above. Therefore, we shall
use this technique to study the homogenization of our model which involves
the time parameter $t$. For more details, see \cite{clark98}.
\end{remark}

\subsection{Homogenization process}

In this subsection, we shall prove Theorem \ref{t1}. The proof is divided
into Lemmata \ref{l2}-\ref{l8}.

\begin{lemma}
\label{l2} There exists a subsequence of $($\textbf{$u$}$^{\varepsilon
},\theta ^{\varepsilon },p^{\varepsilon })$, still denoted $($\textbf{$u$}$%
^{\varepsilon },$ $\theta ^{\varepsilon },p^{\varepsilon })$, and there
exist
\begin{equation*}
\mathbf{u}\in L_{T}^{\infty }(H_{0}^{1}\left( \Omega \right) ^{3}),\quad
\mathbf{\hat{u}}\in L_{T}^{\infty }\left( L^{2}(\Omega ;H_{\#}^{1}(Y)/%
\mathbb{R})^{3}\right) ,\quad p_{m},\theta _{m}\in L_{T}^{\infty
}(H_{0}^{1}(\Omega ))
\end{equation*}%
and
\begin{equation*}
\hat{p}_{m},\hat{\theta}_{m}\in L^{2}(Q;H_{\#}^{1}(Y_{m})/\mathbb{R})
\end{equation*}%
such that, for a.e. $t\in (0,T)$
\begin{eqnarray}
&&\mathbf{u}^{\varepsilon }(t,x)\overset{2-s}{\rightharpoonup }\mathbf{u}%
(t,x)\mathbf{,}  \label{e52} \\
&&\chi _{m}^{\varepsilon }\left( x\right) p_{m}^{\varepsilon }(t,x)\overset{%
2-s}{\rightharpoonup }\chi _{m}\left( y\right) p_{m\,}(t,x),  \label{e53} \\
&&\chi _{m}^{\varepsilon }\left( x\right) \theta _{m}^{\varepsilon }(t,x)%
\overset{2-s}{\rightharpoonup }\chi _{m}\left( y\right) \theta _{m\,}(t,x),
\label{e54} \\
&&\nabla \mathbf{u}^{\varepsilon }(t,x)\overset{2-s}{\rightharpoonup }\nabla
\mathbf{u}(t,x)+\nabla _{y}\mathbf{\hat{u}}(t,x,y),  \label{e55} \\
&&\chi _{m}^{\varepsilon }\left( x\right) \nabla p_{m}^{\varepsilon }(t,x)%
\overset{2-s}{\rightharpoonup }\chi _{m}\left( y\right) (\nabla
p_{m}(t,x)+\nabla _{y}\hat{p}_{m}(t,x,y)),  \label{e56} \\
&&\chi _{m}^{\varepsilon }\left( x\right) \nabla \theta _{m}^{\varepsilon
}(t,x)\overset{2-s}{\rightharpoonup }\chi _{m}\left( y\right) (\nabla \theta
_{m}(t,x)+\nabla _{y}\hat{\theta}_{m}(t,x,y)).  \label{e57}
\end{eqnarray}%
Moreover, for any $\psi \in \mathcal{D}(Q;\mathcal{C}_{\#}(Y))$ we have
\begin{eqnarray}
&&%
\begin{array}{l}
\lim_{\varepsilon \rightarrow 0}\int_{\Sigma _{T}^{\varepsilon }}\zeta
^{\varepsilon }\left( x\right) (p_{1}^{\varepsilon }\left( t,x\right)
-p_{2}^{\varepsilon }\left( t,x\right) )\psi ^{\varepsilon }\left(
t,x\right) \hspace{0.03cm}\mathrm{d}s^{\varepsilon }\left( x\right) \mathrm{d%
}t \\
\\
=\int_{Q\times \Sigma }\varsigma \left( y\right) (p_{1}\left( t,x\right)
-p_{2}\left( t,x\right) )\psi \left( t,x,y\right) \hspace{0.03cm}\mathrm{d}t%
\mathrm{d}x\mathrm{d}s\left( y\right) ,%
\end{array}
\label{e58} \\
&&  \notag \\
&&%
\begin{array}{l}
\lim_{\varepsilon \rightarrow 0}\int_{\Sigma _{T}^{\varepsilon }}\omega
^{\varepsilon }\left( x\right) (\theta _{1}^{\varepsilon }\left( t,x\right)
-\theta _{2}^{\varepsilon }\left( t,x\right) )\psi ^{\varepsilon }\left(
t,x\right) \hspace{0.03cm}\mathrm{d}s^{\varepsilon }\left( x\right) \mathrm{d%
}t \\
\\
=\int_{Q\times \Sigma }\omega \left( y\right) (\theta _{1}\left( t,x\right)
-\theta _{2}\left( t,x\right) )\psi \left( t,x,y\right) \hspace{0.03cm}%
\mathrm{d}t\mathrm{d}x\mathrm{d}s\left( y\right)%
\end{array}
\label{e59}
\end{eqnarray}%
where we have denoted $\psi ^{\varepsilon }(t,x)=\psi (x,t,x/\varepsilon )$.
\end{lemma}

\begin{proof}
The two-scale limits (\ref{e52})-(\ref{e59}) are a straightforward
application of the a priori estimates (\ref{e40}) and the compactness
Theorems \ref{t3}, \ref{t4}, \ref{t5} and \ref{t6}.
\end{proof}

\begin{lemma}
\label{l3}The corrector displacement $\mathbf{\hat{u}}$ can be written as:
\begin{equation}
\mathbf{\hat{u}}(t,x,y)=\mathbf{w}^{ij}(y)e_{ij}(\mathbf{u})(t,x)\text{\ for
a.e. }(t,x,y)\in Q\times Y  \label{e60}
\end{equation}%
where $\mathbf{w}^{ij}\in (H_{\#}^{1}(Y)/\mathbb{R})^{3}$ is the solution to
the microscopic system (\ref{e41}).
\end{lemma}

\begin{proof}
We choose adequate test functions in (\ref{e36}): Let
\begin{equation*}
\mathbf{v}(x):=\mathbf{v}^{\varepsilon }(x)=\varepsilon \mathbf{\hat{v}}(x,%
\dfrac{x}{\varepsilon })
\end{equation*}%
where $\mathbf{\hat{v}}\in \mathcal{D}(\Omega ;\mathcal{C}_{\#}^{\infty
}(Y))^{3}$. Then, we have for a.e. $t\in \left( 0,T\right) $
\begin{equation}
\begin{array}{l}
\int_{\Omega }\mathbf{A}\left( \tfrac{x}{\varepsilon }\right) e(\mathbf{u}%
^{\varepsilon })\left( t,x\right) \left\{ \varepsilon e_{x}\left( \mathbf{%
\hat{v}}\right) (x,\dfrac{x}{\varepsilon })+e_{y}\left( \mathbf{\hat{v}}%
\right) (x,\dfrac{x}{\varepsilon })\right\} \hspace{0.03cm}\mathrm{d}x \\
\\
+\varepsilon \int_{\Omega _{1}^{\varepsilon }}\left( \beta _{1}\nabla
p_{1}^{\varepsilon }\left( t,x\right) +\gamma _{1}\nabla \theta
_{1}^{\varepsilon }\left( t,x\right) \right) \mathbf{\hat{v}}(x,\dfrac{x}{%
\varepsilon })\hspace{0.03cm}\mathrm{d}x \\
\\
+\varepsilon \int_{\Omega _{2}^{\varepsilon }}\left( \beta _{2}\nabla
p_{2}^{\varepsilon }\left( t,x\right) +\gamma _{2}\nabla \theta
_{2}^{\varepsilon }\left( t,x\right) \right) \mathbf{\hat{v}}(x,\dfrac{x}{%
\varepsilon })\hspace{0.03cm}\mathrm{d}x\  \\
\\
=\varepsilon \int_{\Omega _{1}^{\varepsilon }}\mathbf{f}_{1}\left( x\right)
\mathbf{\hat{v}}(x,\dfrac{x}{\varepsilon })\hspace{0.03cm}\mathrm{d}%
x+\varepsilon \int_{\Omega _{2}^{\varepsilon }}\mathbf{f}_{2}\left( x\right)
\mathbf{\hat{v}}(x,\dfrac{x}{\varepsilon })\hspace{0.03cm}\mathrm{d}x.%
\end{array}
\label{e61}
\end{equation}%
In view of (\ref{e55}), we pass to the limit in the first term of the l.h.s.
of (\ref{e61}) to get for a.e. $t\in \left( 0,T\right) $
\begin{eqnarray}
&&\lim_{\varepsilon \rightarrow 0}\int_{\Omega }\mathbf{A}^{\varepsilon
}\left( x\right) e(\mathbf{u}^{\varepsilon })\left( t,x\right) \left\{
\varepsilon e_{x}\left( \mathbf{\hat{v}}\right) (x,\dfrac{x}{\varepsilon }%
)+e_{y}\left( \mathbf{\hat{v}}\right) (x,\dfrac{x}{\varepsilon })\right\}
\hspace{0.03cm}\mathrm{d}x  \notag \\
&=&\int_{\Omega \times Y}\mathbf{A}\left( y\right) \left( e(\mathbf{u}%
)\left( t,x\right) +e_{y}\left( \mathbf{\hat{u}}\right) (t,x,y)\right)
e_{y}\left( \mathbf{\hat{v}}\right) (x,y)\hspace{0.03cm}\mathrm{d}x\mathrm{d}%
y.  \label{e62}
\end{eqnarray}%
Next, since
\begin{equation*}
\left\vert \varepsilon \int_{\Omega _{m}^{\varepsilon }}\beta _{m}\nabla
p_{m}^{\varepsilon }\left( t,x\right) \mathbf{\hat{v}}(x,\dfrac{x}{%
\varepsilon })\hspace{0.03cm}\mathrm{d}x\right\vert \leq \varepsilon \beta
_{m}\left\Vert \mathbf{\hat{v}}(x,\dfrac{x}{\varepsilon }\right\Vert
_{L^{2}\left( \Omega \right) }\left\Vert \nabla p_{m}^{\varepsilon
}\right\Vert _{L^{2}\left( \Omega _{m}^{\varepsilon }\right) },
\end{equation*}%
\begin{equation*}
\left\vert \varepsilon \int_{\Omega _{m}^{\varepsilon }}\gamma _{m}\nabla
\theta _{m}^{\varepsilon }\left( t,x\right) \mathbf{\hat{v}}(x,\dfrac{x}{%
\varepsilon })\hspace{0.03cm}\mathrm{d}x\right\vert \leq \varepsilon \gamma
_{m}\left\Vert \mathbf{\hat{v}}(x,\dfrac{x}{\varepsilon }\right\Vert
_{L^{2}\left( \Omega \right) }\left\Vert \nabla \theta _{m}^{\varepsilon
}\right\Vert _{L^{2}\left( \Omega _{m}^{\varepsilon }\right) }
\end{equation*}%
and taking into account (\ref{e40}) together with Lemma \ref{l1}, we see
that for a.e. $t\in \left( 0,T\right) $
\begin{equation*}
\lim_{\varepsilon \rightarrow 0}\varepsilon \int_{\Omega _{m}^{\varepsilon
}}\left( \beta _{m}\nabla p_{m}^{\varepsilon }\left( t,x\right) +\gamma
_{m}\nabla \theta _{m}^{\varepsilon }\left( t,x\right) \right) \mathbf{\hat{v%
}}(x,\dfrac{x}{\varepsilon })\hspace{0.03cm}\mathrm{d}x=0.
\end{equation*}%
In the same way, letting $\varepsilon \rightarrow 0$ in the r.h.s. of (\ref%
{e61}) we obtain:
\begin{equation}
\lim_{\varepsilon \rightarrow 0}\varepsilon \left( \int_{\Omega
_{1}^{\varepsilon }}\mathbf{f}_{1}\left( x\right) \mathbf{\hat{v}}(x,\dfrac{x%
}{\varepsilon })\hspace{0.03cm}\mathrm{d}x+\int_{\Omega _{2}^{\varepsilon }}%
\mathbf{f}_{2}\left( x\right) \mathbf{\hat{v}}(x,\dfrac{x}{\varepsilon })%
\hspace{0.03cm}\mathrm{d}x\right) =0.  \label{e63}
\end{equation}%
Collecting (\ref{e62}), (\ref{e63}) and passing to the limit in (\ref{e61})
yield for a.e. $t\in \left( 0,T\right) $
\begin{equation*}
\int_{\Omega \times Y}\mathbf{A}\left( y\right) \left( e(\mathbf{u})\left(
t,x\right) +e_{y}\left( \mathbf{\hat{u}}\right) (t,x,y)\right) e_{y}\left(
\mathbf{\hat{v}}\right) (x,y)\hspace{0.03cm}\mathrm{d}x=0
\end{equation*}%
which gives after an integration by parts the following boundary value
problem:
\begin{equation}
\left\{
\begin{array}{ll}
-\mathrm{div}_{y}\left\{ \mathbf{A}\left( y\right) \left( e(\mathbf{u}%
)\left( t,x\right) +e_{y}\left( \mathbf{\hat{u}}\right) (t,x,y)\right)
\right\} =0 & \text{ a.e. in }Q\times Y\text{,} \\
&  \\
y\longmapsto \mathbf{\hat{u}}(t,x,y) & \text{ is }Y-\text{periodic.}%
\end{array}%
\right.  \label{e64}
\end{equation}%
According to the linearity of the system (\ref{e64}), we see that $\mathbf{%
\hat{u}}$ can be written in terms of $\mathbf{u}$ through the following
scale separation expression:%
\begin{equation*}
\mathbf{\hat{u}}(t,x,y)=e_{ij}(\mathbf{u})(t,x)\mathbf{w}^{ij}(y),\ \ \
\text{a.e. }(t,x,y)\in Q\times Y
\end{equation*}%
where $\mathbf{w}^{ij}\in (H_{\#}^{1}(Y)/\mathbb{R})^{3}$ is the solution to
the microscopic system defined by (\ref{e41}), see for instance \cite[page 15%
]{blp}. Hence (\ref{e60}) is proved.
\end{proof}

\begin{lemma}
\label{l4}The corrector pressure $\hat{p}_{m}$ satisfies
\begin{equation}
\hat{p}_{m}(t,x,y)=\pi _{m}^{i}(y)\frac{\partial p_{m}}{\partial x_{i}}%
(t,x),\ \text{a.e. }(t,x,y)\in Q\times Y_{m}  \label{e65}
\end{equation}%
where $\pi _{m}^{i}(y)$ $i=1,2,3,$ $m=1,2$ are defined by (\ref{e42}).
\end{lemma}

\begin{proof}
Let $\hat{q}_{m}\in \mathcal{D}(Q;\mathcal{C}_{\#}^{\infty }(Y))$. Taking
\begin{equation*}
q_{m}\left( t,x\right) :=q_{m}^{\varepsilon }(t,x)=\varepsilon \hat{\varphi}%
_{m}(t,x,\frac{x}{\varepsilon })
\end{equation*}%
in (\ref{e37}) and integrating with respect to $t\in \left( 0,T\right) $, we
get%
\begin{equation}
\begin{array}{l}
\langle \partial _{t}(\phi _{1}p_{1}^{\varepsilon }+\beta _{1}\mathrm{div}%
\mathbf{u}^{\varepsilon }+\alpha _{1}\theta _{1}^{\varepsilon
}),q_{1}^{\varepsilon }\rangle _{V_{1}^{\varepsilon }{}^{\ast
},V_{1}^{\varepsilon }} \\
\\
+\langle \partial _{t}(\phi _{2}p_{2}^{\varepsilon }+\beta _{2}\mathrm{div}%
\mathbf{u}^{\varepsilon }+\alpha _{2}\theta _{2}^{\varepsilon
}),q_{2}^{\varepsilon }\rangle _{V_{2}^{\varepsilon }{}^{\ast
},V_{2}^{\varepsilon }} \\
\\
+\int_{Q_{1}^{\varepsilon }}\kappa _{1}\nabla p_{1}^{\varepsilon }\left(
t,x\right) \left( \varepsilon \nabla _{x}\hat{q}_{1}(t,x,\frac{x}{%
\varepsilon })+\nabla _{y}\hat{q}_{1}(t,x,\frac{x}{\varepsilon })\right)
\hspace{0.03cm}\mathrm{d}t\mathrm{d}x \\
\\
+\int_{Q_{2}^{\varepsilon }}\kappa _{2}\nabla p_{2}^{\varepsilon }\left(
t,x\right) \left( \varepsilon \nabla _{x}\hat{q}_{2}(t,x,\frac{x}{%
\varepsilon })+\nabla _{y}\hat{q}_{2}(t,x,\frac{x}{\varepsilon })\right)
\hspace{0.03cm}\mathrm{d}t\mathrm{d}x \\
\\
+\varepsilon R_{\varepsilon }=0%
\end{array}
\label{e66}
\end{equation}%
where
\begin{equation}
\begin{array}{rl}
R_{\varepsilon }= & -\int_{Q_{1}^{\varepsilon }}(\phi _{1}p_{1}^{\varepsilon
}+\beta _{1}\mathrm{div}\mathbf{u}^{\varepsilon }+\alpha _{1}\theta
_{1}^{\varepsilon })\partial _{t}\hat{q}_{1}^{\varepsilon }\hspace{0.03cm}%
\mathrm{d}t\mathrm{d}x \\
&  \\
& -\int_{Q_{2}^{\varepsilon }}(\phi _{2}p_{2}^{\varepsilon }+\beta _{2}%
\mathrm{div}\mathbf{u}^{\varepsilon }+\alpha _{2}\theta _{2}^{\varepsilon
})\partial _{t}\hat{q}_{2}^{\varepsilon }\hspace{0.03cm}\mathrm{d}t\mathrm{d}%
x \\
&  \\
& +\int_{\Sigma _{T}^{\varepsilon }}\varepsilon \varsigma \left( \frac{x}{%
\varepsilon }\right) (p_{1}^{\varepsilon }-p_{2}^{\varepsilon })(\hat{q}%
_{1}^{\varepsilon }-\hat{q}_{2}^{\varepsilon })\hspace{0.03cm}\mathrm{d}t%
\mathrm{d}s^{\varepsilon }(x) \\
&  \\
& -\int_{Q_{1}^{\varepsilon }}g_{1}\hat{q}_{1}^{\varepsilon }\hspace{0.03cm}%
\mathrm{d}t\mathrm{d}x-\int_{Q_{2}^{\varepsilon }}g_{2}\hat{q}%
_{2}^{\varepsilon }\hspace{0.03cm}\mathrm{d}t\mathrm{d}x%
\end{array}
\label{e67}
\end{equation}%
and $\hat{q}_{m}^{\varepsilon }\left( t,x\right) :=\hat{q}_{m}\left(
x,t,x/\varepsilon \right) $. From (\ref{e58}) it is easy to see that
\begin{equation}
\left\vert \int_{\Sigma _{T}^{\varepsilon }}\varepsilon \varsigma \left(
\frac{x}{\varepsilon }\right) (p_{1}^{\varepsilon }-p_{2}^{\varepsilon })(%
\hat{q}_{1}^{\varepsilon }-\hat{q}_{2}^{\varepsilon })\hspace{0.03cm}\mathrm{%
d}t\mathrm{d}s^{\varepsilon }(x)\right\vert \leq C.  \label{e68}
\end{equation}%
Furthermore, thanks to Lemma \ref{l1} we have
\begin{eqnarray}
&&\lim_{\varepsilon \rightarrow 0}\left( \int_{Q_{1}^{\varepsilon }}g_{1}%
\hat{q}_{1}^{\varepsilon }\hspace{0.03cm}\mathrm{d}t\mathrm{d}%
x+\int_{Q_{2}^{\varepsilon }}g_{2}\hat{q}_{2}^{\varepsilon }\hspace{0.03cm}%
\mathrm{d}t\mathrm{d}x\right) =  \notag \\
&&\int_{Q\times Y_{1}}g_{1}\hat{q}_{1}\mathrm{d}t\mathrm{d}x\mathrm{d}%
y+\int_{Q\times Y_{2}}g_{2}\hat{q}_{2}\hspace{0.03cm}\mathrm{d}t\mathrm{d}x%
\mathrm{d}y.  \label{e69}
\end{eqnarray}%
By virtue of the uniform estimates (\ref{e40}), the sequences $\left\{
p_{m}^{\varepsilon }\right\} _{\varepsilon },\left\{ \mathrm{div}\mathbf{u}%
^{\varepsilon }\right\} _{\varepsilon }$ and $\left\{ \theta
_{m}^{\varepsilon }\right\} _{\varepsilon }$ are uniformly bounded in $%
L^{2}\left( Q_{m}^{\varepsilon }\right) $ so that
\begin{equation}
\left\vert \int_{Q_{m}^{\varepsilon }}(\phi _{m}p_{m}^{\varepsilon }+\beta
_{m}\mathrm{div}\mathbf{u}^{\varepsilon }+\alpha _{m}\theta
_{m}^{\varepsilon })\partial _{t}\hat{q}_{m}^{\varepsilon }\hspace{0.03cm}%
\mathrm{d}t\mathrm{d}x\right\vert \leq C.  \label{e70}
\end{equation}%
Taking into account (\ref{e68})-(\ref{e70}) we get from (\ref{e67}) that
\begin{equation*}
\lim_{\varepsilon \rightarrow 0}\varepsilon R_{\varepsilon }=0\text{.}
\end{equation*}%
On the other hand passing to the limit in (\ref{e66}) and taking into
account (\ref{e56}) we are led to
\begin{eqnarray}
&&\int_{Q\times Y}\chi _{1}\left( y\right) \kappa _{1}\left( \nabla
p_{1}\left( t,x\right) +\nabla _{y}\hat{p}_{1}\left( t,x\right) \right)
\nabla _{y}\hat{q}_{1}(t,x,y)\hspace{0.03cm}\mathrm{d}t\mathrm{d}x\mathrm{d}y
\notag \\
&&+\int_{Q\times Y}\chi _{2}\left( y\right) \kappa _{2}\left( \nabla
p_{2}\left( t,x\right) +\nabla _{y}\hat{p}_{2}\left( t,x,y\right) \right)
\nabla _{y}\hat{q}_{2}(t,x,y)\hspace{0.03cm}\mathrm{d}t\mathrm{d}x=0,  \notag
\end{eqnarray}%
so that an integration by parts yields:%
\begin{equation}
\left\{
\begin{array}{l}
-\mathrm{div}\left( \kappa _{1}\left( \nabla p_{1}\left( t,x\right) +\nabla
_{y}\hat{p}_{1}\left( t,x,y\right) \right) \right) =0\text{ a.e. in }Q\times
Y_{1}, \\
\\
-\mathrm{div}\left( \kappa _{2}\left( \nabla p_{2}\left( t,x\right) +\nabla
_{y}\hat{p}_{2}\left( t,x,y\right) \right) \right) =0\text{ a.e. in }Q\times
Y_{2}, \\
\\
\left[ \kappa _{m}\left( \nabla p_{m}\left( t,x\right) +\nabla _{y}\hat{p}%
_{m}\left( t,x,y\right) \right) \right] \cdot \nu \left( y\right) \text{
a.e. on }Q\times \Sigma , \\
\\
y\longmapsto \hat{p}_{1},\hat{p}_{2}\text{ are }Y-\text{periodic.}%
\end{array}%
\right.  \label{e71}
\end{equation}%
As in Lemma \ref{l3} and thanks to the linearity of the system (\ref{e71})
we can write that
\begin{equation*}
\hat{p}_{m}(t,x,y)=\pi _{m}^{i}(y)\frac{\partial p_{m}}{\partial x_{i}}%
(t,x)\ \text{a.e. }(t,x,y)\in Q\times Y_{m}
\end{equation*}%
where $\pi _{m}^{i}(y)\in \pi _{m\ i}\in H_{\#}^{1}(Y_{m})/\mathbb{R}$ is
the solution of (\ref{e42}). The Lemma is then proved.
\end{proof}

\begin{lemma}
\label{l5}The corrector temperature $\hat{\theta}_{m}$ is related to the
homogenized temperature $\theta _{m}$ via the linear relation:
\begin{equation}
\hat{\theta}_{m}\left( t,x,y\right) =\vartheta _{m}^{i}(y)\frac{\partial
\theta _{m}}{\partial x_{i}}\left( t,x\right) +C^{te},\ \ \text{a.e. }%
(t,x,y)\in Q\times Y_{m}  \label{e72}
\end{equation}%
where, for $i=1,2,3$, the "micro-temperature" $\vartheta _{m}^{i}\in
H^{1}(Y_{m})/\mathbb{R}$ is the solution of (\ref{e43}).
\end{lemma}

The proof of this Lemma follows the same lines as that of Lemma \ref{l4} and
therefore will not be given.

\begin{lemma}
\label{l6}The macroscopic balance equation reads as follows:
\begin{equation}
\left\{
\begin{array}{l}
-\mathrm{div}\left( \sigma ^{0}\left( \mathbf{u}\right) \right) +B_{1}\nabla
p_{1}+B_{2}\nabla p_{2}+D_{1}\nabla \theta _{1}+D_{2}\nabla \theta _{2}=%
\mathbf{f}\text{ in }\Omega \text{,} \\
\  \\
\mathbf{u}=0\text{ on }\partial \Omega%
\end{array}%
\right.  \label{e73}
\end{equation}%
where $\mathcal{A},\ B_{m},\ D_{m}$ and $\mathbf{f}$ are defined by (\ref%
{e45}), (\ref{e46}), (\ref{e47}) and (\ref{e50}) respectively.
\end{lemma}

\begin{proof}
The convergence results obtained in Lemma \ref{l2} allow us to derive the
macroscopic equations (\ref{e51}). To do so, we first determine the limiting
equations of (\ref{e36})-(\ref{e39}). Let $\mathbf{v\in }\mathcal{D}(\Omega
)^{3}$. Multiplying (\ref{e36}) by $\mathbf{v}$ and passing to the limit, we
find by virtue of (\ref{e55})-(\ref{e57}) and Remark \ref{r1} that for a.e. $%
t\in \left( 0,T\right) $
\begin{eqnarray}
&&\int_{\Omega \times Y}\mathbf{A}[e(\mathbf{u})+e_{y}(\mathbf{\hat{u}})]e(%
\mathbf{v})\hspace{0.03cm}\mathrm{d}x\mathrm{d}y+\beta _{1}\int_{\Omega
\times Y_{1}}\left( \nabla p_{1}+\nabla _{y}\hat{p}_{1}\right) \mathbf{v}%
\hspace{0.03cm}\mathrm{d}x\mathrm{d}y  \notag \\
&&+\beta _{2}\int_{\Omega \times Y_{2}}\left( \nabla p_{2}+\nabla _{y}\hat{p}%
_{2}\right) \mathbf{v}\hspace{0.03cm}\mathrm{d}x\mathrm{d}y+\gamma
_{1}\int_{\Omega \times Y_{1}}\left( \nabla \theta _{1}+\nabla _{y}\hat{%
\theta}_{1}\right) \mathbf{v}\hspace{0.03cm}\mathrm{d}x\mathrm{d}y  \notag \\
&&+\gamma _{2}\int_{\Omega \times Y_{2}}\left( \nabla \theta _{2}+\nabla _{y}%
\hat{\theta}_{2}\right) \mathbf{v}\hspace{0.03cm}\mathrm{d}x\mathrm{d}y
\notag \\
&=&\int_{\Omega \times Y_{1}}\mathbf{f}_{1}\mathbf{v}\hspace{0.03cm}\mathrm{d%
}x\mathrm{d}y+\int_{\Omega \times Y_{2}}\mathbf{f}_{2}\mathbf{v}\hspace{%
0.03cm}\mathrm{d}x\mathrm{d}y.  \label{e74}
\end{eqnarray}%
Let us rewrite the first integral term in the l.h.s. of (\ref{e74}) with the
help of (\ref{e44}), (\ref{e45}) and (\ref{e60}). We have%
\begin{equation}
\int_{\Omega \times Y}\mathbf{A}_{ijkh}[e(\mathbf{u})+e_{y}(\mathbf{\hat{u}}%
)]e(\mathbf{v})\hspace{0.03cm}\mathrm{d}x\mathrm{d}y=\int_{\Omega }\sigma
_{ij}^{\hom }\left( \mathbf{u}\right) e_{ij}(\mathbf{v})\hspace{0.03cm}%
\mathrm{d}x.  \label{e75}
\end{equation}%
Likewise, by using (\ref{e65}) and (\ref{e72}), there holds
\begin{eqnarray}
&&\beta _{1}\int_{\Omega \times Y_{1}}\left( \nabla p_{1}+\nabla _{y}\hat{p}%
_{1}\right) \mathbf{v}\hspace{0.03cm}\mathrm{d}x+\beta _{2}\int_{\Omega
\times Y_{2}}\left( \nabla p_{2}+\nabla _{y}\hat{p}_{2}\right) \mathbf{v}%
\hspace{0.03cm}\mathrm{d}x\mathrm{d}y  \notag \\
&=&\beta _{1}\int_{\Omega \times Y}\chi _{1}\left( \delta _{ik}+\frac{%
\partial \pi _{1}^{i}}{\partial y_{k}}\right) \frac{\partial p_{1}}{\partial
x_{i}}v_{k}\hspace{0.03cm}\mathrm{d}x  \notag \\
&&+\beta _{2}\int_{\Omega \times Y}\chi _{2}\left( \delta _{ik}+\frac{%
\partial \pi _{2}^{i}}{\partial y_{k}}\right) \frac{\partial p_{2}}{\partial
x_{i}}v_{k}\hspace{0.03cm}\mathrm{d}x  \label{e76}
\end{eqnarray}%
where $v_{k}$ is the $k^{th}$ component of $\mathbf{v}$. After simple
algebraic calculations, (\ref{e76}) becomes then
\begin{eqnarray}
&&\beta _{1}\int_{\Omega \times Y_{1}}\left( \nabla p_{1}+\nabla _{y}\hat{p}%
_{1}\right) \mathbf{v}\hspace{0.03cm}\mathrm{d}x\mathrm{d}y+\beta
_{2}\int_{\Omega \times Y_{2}}\left( \nabla p_{2}+\nabla _{y}\hat{p}%
_{2}\right) \mathbf{v}\hspace{0.03cm}\mathrm{d}x\mathrm{d}y  \notag \\
&=&\int_{\Omega }B_{1}\nabla p_{1}\left( x\right) \mathbf{v}\hspace{0.03cm}%
\mathrm{d}x+\int_{\Omega }B_{2}\nabla p_{2}\mathbf{v}\hspace{0.03cm}\mathrm{d%
}x.  \label{e77}
\end{eqnarray}%
In the same way, one can show that
\begin{eqnarray}
&&\int_{\Omega \times Y_{1}}\gamma _{1}\left( \nabla \theta _{1}+\nabla _{y}%
\hat{\theta}_{1}\right) \mathbf{v}\hspace{0.03cm}\mathrm{d}x\mathrm{d}%
y+\gamma _{2}\int_{\Omega \times Y_{2}}\left( \nabla \theta _{2}+\nabla _{y}%
\hat{\theta}_{2}\right) \mathbf{v}\hspace{0.03cm}\mathrm{d}x\mathrm{d}y
\notag \\
&=&\int_{\Omega }D_{1}\nabla \theta _{1}\mathbf{v}\hspace{0.03cm}\mathrm{d}%
x+\int_{\Omega }D_{2}\nabla \theta _{2}\mathbf{v}\hspace{0.03cm}\mathrm{d}x.
\label{e78}
\end{eqnarray}%
Finally, inserting (\ref{e50}), (\ref{e75}), (\ref{e77}) and (\ref{e78})
into (\ref{e74}) and using the fact that $\mathcal{D}(\Omega )^{3}$ is dense
in $H_{0}^{1}\left( \Omega \right) ^{3}$, we deduce the homogenized balance
formulation:
\begin{equation}
\begin{array}{c}
\int_{\Omega }\mathcal{A}_{ijkh}e_{kh}(\mathbf{u})e_{ij}(\mathbf{v})\hspace{%
0.03cm}\mathrm{d}x+\int_{\Omega }B_{1}\nabla p_{1}\mathbf{v}\hspace{0.03cm}%
\mathrm{d}x+\int_{\Omega }B_{2}\nabla p_{2}\mathbf{v}\hspace{0.03cm}\mathrm{d%
}x \\
\\
+\int_{\Omega }D_{1}\nabla \theta _{1}\mathbf{v}\hspace{0.03cm}\mathrm{d}%
x+\int_{\Omega }D_{2}\nabla \theta _{2}\mathbf{v}\hspace{0.03cm}\mathrm{d}%
x=\int_{\Omega }\mathbf{fv}\hspace{0.03cm}\mathrm{d}x%
\end{array}
\label{e79}
\end{equation}%
which by an integration by parts yields (\ref{e73}). The Lemma is then
proved.
\end{proof}

\begin{lemma}
\label{l7}The macroscopic mass conservation equation is given by:%
\begin{equation}
\left\{
\begin{array}{l}
\partial _{t}\left( \phi _{1}^{\ast }p_{1}+\beta _{1}\mathcal{C}_{1}:e\left(
\mathbf{u}\right) +\alpha _{1}^{\ast }\theta _{1}\right) -\mathrm{div}\left(
\mathcal{K}_{1}\nabla p_{1}\right) +\zeta ^{\ast }\left( p_{1}-p_{2}\right)
=g_{1}^{\ast },\text{ in }Q, \\
\\
\partial _{t}\left( \phi _{2}^{\ast }p_{2}+\beta _{2}\mathcal{C}_{2}:e\left(
\mathbf{u}\right) +\alpha _{2}^{\ast }\theta _{2}\right) -\mathrm{div}\left(
\mathcal{K}_{2}\nabla p_{2}\right) +\zeta ^{\ast }\left( p_{2}-p_{1}\right)
=g_{2}^{\ast },\text{ in }Q, \\
\\
p_{1}=0,\text{ on }\Gamma _{T}, \\
\\
\mathcal{K}_{2}\nabla p_{2}\cdot \nu =0\text{ on }\Gamma _{T}, \\
\\
p_{1}\left( 0,.\right) =p_{2}\left( 0,.\right) =0\text{ in }\Omega%
\end{array}%
\right.  \label{e80}
\end{equation}%
where $\mathcal{C}_{m},\ \mathcal{K}_{m}$ and $\left( \phi _{m}^{\ast
},\alpha _{m}^{\ast },\zeta ^{\ast },g_{m}^{\ast }\right) $ are given in (%
\ref{e48}), (\ref{e49}) and (\ref{e50}) respectively.
\end{lemma}

\begin{proof}
Let $q_{m}(t,x)\in \mathcal{D}((0,T)\times \Omega )$. Integration by parts
in (\ref{e37}) with respect to the time variable $t\in \left( 0,T\right) $
yields:
\begin{equation}
\begin{array}{l}
\int_{Q_{1}^{\varepsilon }}(\phi _{1}p_{1}^{\varepsilon }\left( t,x\right)
+\beta _{1}\mathrm{div}\mathbf{u}^{\varepsilon }\left( t,x\right) +\alpha
_{1}\theta _{1}^{\varepsilon }\left( t,x\right) )\partial _{t}q_{1}\left(
t,x\right) \hspace{0.03cm}\mathrm{d}t\mathrm{d}x \\
\\
+\int_{Q_{1}^{\varepsilon }}(\phi _{2}p_{2}^{\varepsilon }\left( t,x\right)
+\beta _{2}\mathrm{div}\mathbf{u}^{\varepsilon }\left( t,x\right) +\alpha
_{2}\theta _{2}^{\varepsilon }\left( t,x\right) )\partial _{t}q_{2}\left(
t,x\right) \hspace{0.03cm}\mathrm{d}t\mathrm{d}x \\
\\
+\int_{Q_{1}^{\varepsilon }}\kappa _{1}\nabla p_{1}^{\varepsilon }\left(
t,x\right) \nabla q_{1}\left( t,x\right) \hspace{0.03cm}\mathrm{d}t\mathrm{d}%
x+\int_{Q_{2}^{\varepsilon }}\kappa _{2}\nabla p_{2}^{\varepsilon }\left(
t,x\right) \nabla q_{2}\left( t,x\right) \hspace{0.03cm}\mathrm{d}t\mathrm{d}%
x \\
\\
+\int_{\Sigma _{T}^{\varepsilon }}\varsigma ^{\varepsilon }\left( x\right)
(p_{1}^{\varepsilon }\left( t,x\right) -p_{2}^{\varepsilon }\left(
t,x\right) )(q_{1}\left( t,x\right) -q_{2}\left( t,x\right) )\hspace{0.03cm}%
\mathrm{d}s^{\varepsilon }(x)\mathrm{d}t \\
\\
=\int_{Q_{1}^{\varepsilon }}g_{1}\left( x\right) q_{1}\left( t,x\right)
\hspace{0.03cm}\mathrm{d}t\mathrm{d}x+\int_{Q_{2}^{\varepsilon }}g_{2}\left(
x\right) q_{2}\left( t,x\right) \hspace{0.03cm}\mathrm{d}t\mathrm{d}x.%
\end{array}
\label{e81}
\end{equation}%
Using (\ref{e53}) and (\ref{e54}) we obtain:
\begin{eqnarray}
&&\underset{\varepsilon \rightarrow 0}{\lim }\int_{Q_{m}^{\varepsilon
}}(\phi _{m}p_{m}^{\varepsilon }\left( t,x\right) +\alpha _{m}\theta
_{m}^{\varepsilon }\left( t,x\right) )\partial _{t}q_{m}\left( t,x\right)
\hspace{0.03cm}\mathrm{d}t\mathrm{d}x  \notag \\
&=&\int_{Q\times Y_{m}}(\phi _{m}p_{m}\left( t,x\right) +\alpha _{m}\theta
_{m}\left( t,x\right) )\partial _{t}q_{m}\left( t,x\right) \hspace{0.03cm}%
\mathrm{d}t\mathrm{d}x\mathrm{d}y.  \label{e82}
\end{eqnarray}%
Furthermore, thanks to (\ref{e55}) and (\ref{e56}) we see that%
\begin{equation}
\begin{array}{l}
\underset{\varepsilon \rightarrow 0}{\lim }\int_{Q_{m}^{\varepsilon }}\beta
_{m}\mathrm{div}\mathbf{u}^{\varepsilon }\left( t,x\right) \partial
_{t}q_{m}\left( t,x\right) \hspace{0.03cm}\mathrm{d}t\mathrm{d}x \\
\  \\
=\int_{Q\times Y_{m}}\beta _{m}\left( \mathrm{div}\mathbf{u}\left(
t,x\right) +\mathrm{div}_{y}\mathbf{\hat{u}}\left( t,x,y\right) \right)
\partial _{t}q_{m}\left( t,x\right) \hspace{0.03cm}\mathrm{d}t\mathrm{d}x%
\mathrm{d}y%
\end{array}
\label{e83}
\end{equation}%
and
\begin{equation}
\begin{array}{l}
\underset{\varepsilon \rightarrow 0}{\lim }\int_{Q_{m}^{\varepsilon }}\kappa
_{m}\nabla p_{m}^{\varepsilon }\left( t,x\right) \nabla q_{m}\left(
t,x\right) \hspace{0.03cm}dx\mathrm{d}t \\
\\
=\int_{Q\times Y_{m}}(\kappa _{m}\left\{ \nabla p_{m}\left( t,x\right)
+\nabla _{y}\hat{p}_{m}\left( t,x,y\right) \right\} )\nabla q_{m}\left(
t,x\right) \hspace{0.03cm}\mathrm{d}t\mathrm{d}x\mathrm{d}y.%
\end{array}
\label{e84}
\end{equation}%
Now, using (\ref{e60}) we observe that
\begin{equation}
\begin{array}{l}
\int_{Q\times Y_{m}}\beta _{m}\left( \mathrm{div}\mathbf{u}\left( t,x\right)
+\mathrm{div}_{y}\mathbf{\hat{u}}\left( t,x,y\right) \right) \partial
_{t}q_{m}\left( t,x\right) \hspace{0.03cm}\mathrm{d}t\mathrm{d}x\mathrm{d}y
\\
\\
=\int_{Q}\beta _{m}\mathcal{C}_{m}:e\left( \mathbf{u}\right) \left(
t,x\right) \partial _{t}q_{m}\left( t,x\right) \hspace{0.03cm}\mathrm{d}t%
\mathrm{d}x.%
\end{array}
\label{e85}
\end{equation}
In the same way, using (\ref{e65}) we find that
\begin{equation}
\begin{array}{l}
\int_{Q\times Y_{m}}(\kappa _{m}\left\{ \nabla p_{m}\left( t,x\right)
+\nabla _{y}\hat{p}_{m}\left( t,x,y\right) \right\} )\nabla q_{m}\left(
t,x\right) \hspace{0.03cm}\mathrm{d}x\mathrm{d}y\mathrm{d}t \\
\\
=\int_{Q}\mathcal{K}_{m}\nabla p_{m}\left( x,t\right) \nabla q_{m}\left(
x,t\right) \hspace{0.03cm}\mathrm{d}t\mathrm{d}x.%
\end{array}
\label{e86}
\end{equation}
Now, taking into account (\ref{e58}) we have
\begin{equation}
\begin{array}{c}
\underset{\varepsilon \rightarrow 0}{\lim }\int_{\Sigma _{T}^{\varepsilon
}}\varepsilon \varsigma \left( \frac{x}{\varepsilon }\right)
(p_{1}^{\varepsilon }\left( x,t\right) -p_{2}^{\varepsilon }\left( x\right)
)(q_{1}\left( x,t\right) -q_{2}\left( x,t\right) )\hspace{0.03cm}\mathrm{d}%
s^{\varepsilon }(x)\mathrm{d}t \\
\  \\
=\int_{Q\times \Sigma }\varsigma \left( y\right) (p_{1}\left( x,t\right)
-p_{2}\left( x\right) )(q_{1}\left( t,x\right) -q_{2}\left( t,x\right) )%
\hspace{0.03cm}\mathrm{d}t\mathrm{d}x\mathrm{d}s(y).%
\end{array}
\label{e87}
\end{equation}%
Using Remark \ref{r1}, we infer that
\begin{equation}
\begin{array}{l}
\underset{\varepsilon \rightarrow 0}{\lim }\left( \int_{Q_{1}^{\varepsilon
}}g_{1}\left( x\right) q_{1}\left( t,x\right) \hspace{0.03cm}\mathrm{d}t%
\mathrm{d}x+\int_{Q_{2}^{\varepsilon }}g_{2}\left( x\right) q_{2}\left(
t,x\right) \hspace{0.03cm}\mathrm{d}t\mathrm{d}x\right) \\
\\
=\int_{Q}\tilde{g}_{1}q_{1}\left( t,x\right) \hspace{0.03cm}\mathrm{d}t%
\mathrm{d}x+\int_{Q}\tilde{g}_{2}q_{2}\left( t,x\right) \hspace{0.03cm}%
\mathrm{d}t\mathrm{d}x.%
\end{array}
\label{e88}
\end{equation}%
Finally, having in mind (\ref{e82})-(\ref{e88}), we can now pass to the
limit in (\ref{e81}) to get
\begin{equation*}
\begin{array}{l}
\int_{Q}\left\{ (\phi _{1}^{\ast }p_{1}\left( t,x\right) +\beta _{1}\mathcal{%
C}_{1}:e\left( \mathbf{u}\right) \left( t,x\right) +\alpha _{1}^{\ast
}\theta _{1}\left( t,x\right) \right\} \partial _{t}q_{1}\left( t,x\right)
\hspace{0.03cm}\mathrm{d}t\mathrm{d}x \\
+\int_{Q}\left\{ (\phi _{2}^{\ast }p_{2}\left( t,x\right) +\beta _{2}%
\mathcal{C}_{2}:e\left( \mathbf{u}\right) \left( t,x\right) +\alpha
_{2}^{\ast }\theta _{2}\left( t,x\right) \right\} \partial _{t}q_{2}\left(
t,x\right) \hspace{0.03cm}\mathrm{d}t\mathrm{d}x \\
+\int_{Q}\mathcal{K}_{1}\nabla p_{1}\left( t,x\right) \nabla q_{1}\left(
t,x\right) \hspace{0.03cm}\mathrm{d}t\mathrm{d}x+\int_{Q}\mathcal{K}%
_{2}\nabla p_{2}\left( t,x\right) \nabla q_{2}\left( t,x\right) \hspace{%
0.03cm}\mathrm{d}t\mathrm{d}x \\
+\int_{Q}\zeta ^{\ast }(p_{1}\left( t,x\right) -p_{2}\left( t,x\right)
)(q_{1}\left( t,x\right) -q_{2}\left( t,x\right) )\hspace{0.03cm}\mathrm{d}t%
\mathrm{d}x \\
=\int_{Q}g_{1}^{\ast }\left( x\right) q_{1}\left( t,x\right) \hspace{0.03cm}%
\mathrm{d}t\mathrm{d}x+\int_{Q}g_{2}^{\ast }\left( x\right) q_{2}\left(
t,x\right) \hspace{0.03cm}\mathrm{d}t\mathrm{d}x%
\end{array}%
\end{equation*}%
which by integration by parts yields (\ref{e80}). This gives the desired
result.
\end{proof}

Likewise, as in the proof of Lemma \ref{l7}, one can analogously show the
following result:

\begin{lemma}
\label{l8}The macroscopic heat equation is given by:%
\begin{equation*}
\left\{
\begin{array}{l}
\partial _{t}\left\{ c_{1}^{\ast }\theta _{1}+\gamma _{1}\mathcal{C}%
_{1}:e\left( \mathbf{u}\right) +\alpha _{1}^{\ast }p_{1}\right\} -\mathrm{div%
}\left( \mathcal{L}_{1}\nabla \theta _{1}\right) +\omega ^{\ast }\left(
\theta _{1}-\theta _{2}\right) =h_{1}^{\ast },\text{ in }Q, \\
\\
\partial _{t}\left\{ c_{2}^{\ast }\theta _{2}+\gamma _{2}\mathcal{C}%
_{2}:e\left( \mathbf{u}\right) +\alpha _{2}^{\ast }p_{2}\right\} -\mathrm{div%
}\left( \mathcal{L}_{2}\nabla \theta _{2}\right) +\omega ^{\ast }\left(
\theta _{2}-\theta _{1}\right) =h_{2}^{\ast },\text{ in }Q, \\
\\
\theta _{1}=0,\text{ on }\Gamma _{T}, \\
\\
\mathcal{L}_{2}\nabla \theta _{2}\cdot \nu =0\text{ on }\Gamma _{T}, \\
\\
\theta _{1}\left( 0,.\right) =\theta _{2}\left( 0,.\right) =0\text{ in }%
\Omega%
\end{array}%
\right.
\end{equation*}%
where $\mathcal{L}_{m}$ and $\left( c_{m}^{\ast },\omega ^{\ast
},h_{m}^{\ast }\right) $ are given in (\ref{e49}) and (\ref{e50})
respectively.
\end{lemma}

\

\begin{proof}[Proof of Theorem \ref{t1}]
Collecting all the results given in Lemmata \ref{l6}, \ref{l7}, \ref{l8}, we arrive at
the homogenized system (\ref{e51}).
\end{proof}

\section{A corrector result\label{cr}}

Let $\mathbf{\hat{u}}$, $\hat{p}_{m}$ and $\hat{\theta}_{m}$ to be the
corrector terms of \textbf{$u$}$^{\varepsilon }$, $p_{m}^{\varepsilon }$ and
$\theta _{m}^{\varepsilon }$ respectively. They are given by (\ref{e55}), (%
\ref{e56}) and (\ref{e57}) respectively. Now we first give the following
corrector result.

\begin{theorem}
\label{t2}Assume that $A\in L^{\infty }\left( Y\right) $. Then $e_{y}(\hat{u}%
)$\ is an admissible test function and the sequence $[e($\textbf{$u$}$%
^{\varepsilon })\left( x\right) -e($\textbf{$u$}$)\left( x\right) -e_{y}(%
\mathbf{\hat{u}})\left( x,\frac{x}{\varepsilon }\right) ]$ converges
strongly to $0$ in $L^{2}\left( \Omega \right) ^{3\times 3}$. We also have
\begin{equation*}
\begin{array}{l}
\lim_{\varepsilon \rightarrow 0}\left\Vert \chi _{m}\left( \frac{x}{%
\varepsilon }\right) \left\{ \nabla p_{m}^{\varepsilon }\left( t,x\right)
-\nabla p_{m}\left( t,x\right) -\nabla _{y}\hat{p}_{m}\left( t,x,\frac{x}{%
\varepsilon }\right) \right\} \right\Vert _{0,Q_{T}}=0, \\
\  \\
\lim_{\varepsilon \rightarrow 0}\left\Vert \chi _{m}\left( \frac{x}{%
\varepsilon }\right) \left\{ \nabla \theta _{m}^{\varepsilon }\left(
t,x\right) -\nabla \theta _{m}\left( t,x\right) -\nabla _{y}\hat{\theta}%
_{m}\left( t,x,\frac{x}{\varepsilon }\right) \right\} \right\Vert
_{0,Q_{T}}=0%
\end{array}%
\end{equation*}%
and
\begin{equation*}
\begin{array}{r}
\lim_{\varepsilon }\left( \int_{\Omega }\mathbf{A}^{\varepsilon }e(\mathbf{u}%
^{\varepsilon })e(\mathbf{u}^{\varepsilon })\hspace{0.03cm}\mathrm{d}%
x+\sum_{m}\int_{\Omega _{m}^{\varepsilon }}\left( \alpha _{m}\nabla
p_{m}^{\varepsilon }+\gamma _{m}\nabla \theta _{m}^{\varepsilon }\right)
\mathbf{u}^{\varepsilon }\hspace{0.03cm}\mathrm{d}x\right) \\
\\
=\int_{\Omega }\mathcal{A}e(\mathbf{u})e(\mathbf{u})\hspace{0.03cm}\mathrm{d}%
x+\sum_{m}\int_{\Omega }\left( \xi _{m}\nabla \theta _{m}+\eta _{m}\nabla
p_{m}\right) \mathbf{u}^{0}\hspace{0.03cm}\mathrm{d}x.%
\end{array}%
\end{equation*}
\end{theorem}

To prove Theorem \ref{t2}, we begin by establishing an integral identity.

\begin{proposition}
We have
\begin{eqnarray}
&&-\int_{\Omega }\phi _{1}^{\ast }(p_{1})^{2}\left( T,x\right) \hspace{0.03cm%
}\mathrm{d}x-\int_{Q\times Y_{1}}\left[ \partial _{t}(\beta _{1}\left(
\mathrm{div}\mathbf{u}+\mathrm{div}_{y}\mathbf{\hat{u}}\right) +\alpha
_{1}\theta _{1})\right] p_{1}\hspace{0.03cm}\mathrm{d}t\mathrm{d}x\mathrm{d}y
\notag \\
&&-\int_{\Omega }\phi _{2}^{\ast }(p_{2})^{2}\left( T,x\right) \hspace{0.03cm%
}\mathrm{d}x-\int_{Q\times Y_{2}}\left[ \partial _{t}(\beta _{2}\left(
\mathrm{div}\mathbf{u}+\mathrm{div}_{y}\mathbf{\hat{u}}\right) +\alpha
_{2}\theta _{2})\right] p_{2}\hspace{0.03cm}\mathrm{d}t\mathrm{d}x\mathrm{d}y
\notag \\
&&+\int_{Q\times Y_{1}}\mathcal{K}_{1}\nabla p_{1}\nabla p_{1}\hspace{0.03cm}%
\mathrm{d}t\mathrm{d}x\mathrm{d}y+\int_{Q\times Y_{2}}\mathcal{K}_{2}\nabla
p_{2}\nabla p_{2}\hspace{0.03cm}\mathrm{d}t\mathrm{d}x\mathrm{d}y  \notag \\
&&+\int_{Q}\varsigma ^{\ast }\left( p_{1}-p_{2}\right) ^{2}\hspace{0.03cm}%
\mathrm{d}t\mathrm{d}x=\int_{\Omega }g_{1}^{\ast }p_{1}\hspace{0.03cm}%
\mathrm{d}x+\int_{\Omega }g_{2}^{\ast }p_{2}\hspace{0.03cm}\mathrm{d}x
\label{e89}
\end{eqnarray}

and%
\begin{eqnarray}
&&-\int_{\Omega }c_{1}^{\ast }(\theta _{1})^{2}\left( T,x\right) \hspace{%
0.03cm}\mathrm{d}x\mathrm{d}y-\int_{Q\times Y_{m}}\left[ \partial
_{t}(\gamma _{1}\left( \mathrm{div}\mathbf{u}+\mathrm{div}_{y}\mathbf{\hat{u}%
}\right) +\alpha _{1}p_{1})\right] \theta _{1}\hspace{0.03cm}\mathrm{d}t%
\mathrm{d}x\mathrm{d}y  \notag \\
&&-\int_{\Omega \times Y_{2}}c_{2}^{\ast }(\theta _{2})^{2}\left( T,x\right)
\hspace{0.03cm}\mathrm{d}x\mathrm{d}y-\int_{Q\times Y_{m}}\left[ \partial
_{t}(\gamma _{2}\left( \mathrm{div}\mathbf{u}+\mathrm{div}_{y}\mathbf{\hat{u}%
}\right) +\alpha _{2}p_{2})\right] \theta _{2}\hspace{0.03cm}\mathrm{d}t%
\mathrm{d}x\mathrm{d}y  \notag \\
&&+\int_{Q\times Y_{1}}\mathcal{L}_{1}\nabla \theta _{1}\nabla \theta _{1}%
\hspace{0.03cm}\mathrm{d}t\mathrm{d}x\mathrm{d}y+\int_{Q\times Y_{2}}%
\mathcal{L}_{2}\nabla \theta _{2}\nabla \theta _{2}\hspace{0.03cm}\mathrm{d}t%
\mathrm{d}x\mathrm{d}y+  \notag \\
&&\int_{Q}\omega ^{\ast }\left( \theta _{1}-\theta _{2}\right) ^{2}\hspace{%
0.03cm}\mathrm{d}t\mathrm{d}x=\int_{\Omega }h_{1}^{\ast }\theta _{1}\hspace{%
0.03cm}\mathrm{d}x+\int_{\Omega }h_{2}^{\ast }\theta _{2}\hspace{0.03cm}%
\mathrm{d}x.  \label{e90}
\end{eqnarray}
\end{proposition}

\begin{proof}
Take $q_{m}\in L^{2}\left( 0,T;H^{1}\left( \Omega _{m}^{\varepsilon }\right)
\right) $, $m=1,2$ in (\ref{e37}) and integrate by parts to yield:
\begin{equation*}
\begin{array}{l}
-\int_{Q_{1}^{\varepsilon }}(\phi _{1}p_{1}^{\varepsilon }\left( t,x\right)
+\beta _{1}\mathrm{div}\mathbf{u}^{\varepsilon }\left( T,x\right) +\alpha
_{1}\theta _{1}^{\varepsilon }\left( t,x\right) )\partial _{t}q_{1}\left(
t,x\right) \hspace{0.03cm}\mathrm{d}t\mathrm{d}x \\
\\
-\int_{Q_{2}^{\varepsilon }}(\phi _{2}p_{2}^{\varepsilon }\left( t,x\right)
+\beta _{2}\mathrm{div}\mathbf{u}^{\varepsilon }\left( t,x\right) +\alpha
_{2}\theta _{2}^{\varepsilon }\left( t,x\right) )\partial _{t}q_{2}\left(
t,x\right) \hspace{0.03cm}\mathrm{d}t\mathrm{d}x \\
\\
+\int_{\Omega _{1}^{\varepsilon }}(\phi _{1}p_{1}^{\varepsilon }\left(
T,x\right) +\beta _{1}\mathrm{div}\mathbf{u}^{\varepsilon }\left( T,x\right)
+\alpha _{1}\theta _{1}^{\varepsilon }\left( T,x\right) )q_{1}\left(
T,x\right) \hspace{0.03cm}\mathrm{d}x \\
\\
+\int_{\Omega _{2}^{\varepsilon }}(\phi _{2}p_{2}^{\varepsilon }\left(
T,x\right) +\beta _{2}\mathrm{div}\mathbf{u}^{\varepsilon }\left( T,x\right)
+\alpha _{2}\theta _{2}^{\varepsilon }\left( T,x\right) )q_{2}\left(
T,x\right) \hspace{0.03cm}\mathrm{d}t\mathrm{d}x \\
\\
-\int_{\Omega _{1}^{\varepsilon }}(\phi _{1}p_{1}^{\varepsilon }\left(
0,x\right) +\beta _{1}\mathrm{div}\mathbf{u}^{\varepsilon }\left( 0,x\right)
+\alpha _{1}\theta _{1}^{\varepsilon }\left( 0,x\right) )q_{1}\left(
0,x\right) \hspace{0.03cm}\mathrm{d}x \\
\\
-\int_{\Omega _{2}^{\varepsilon }}(\phi _{2}p_{2}^{\varepsilon }\left(
0,x\right) +\beta _{2}\mathrm{div}\mathbf{u}^{\varepsilon }\left( 0,x\right)
+\alpha _{2}\theta _{2}^{\varepsilon }\left( 0,x\right) )q_{2}\left(
0,x\right) \hspace{0.03cm}\mathrm{d}x \\
\\
\int_{Q_{1}^{\varepsilon }}\kappa _{1}\nabla p_{1}^{\varepsilon }\left(
t,x\right) \nabla q_{1}\left( t,x\right) \hspace{0.03cm}\mathrm{d}t\mathrm{d}%
x+\int_{Q_{2}^{\varepsilon }}\kappa _{2}\nabla p_{2}^{\varepsilon }\left(
t,x\right) \nabla q_{2}\left( t,x\right) \hspace{0.03cm}\mathrm{d}t\mathrm{d}%
x \\
\\
+\int_{\Sigma _{T}^{\varepsilon }}\varsigma ^{\varepsilon }\left( x\right)
(p_{1}^{\varepsilon }\left( t,x\right) -p_{2}^{\varepsilon }\left(
t,x\right) )(q_{1}\left( t,x\right) -q_{2}\left( t,x\right) )\hspace{0.03cm}%
\mathrm{d}s^{\varepsilon }(x)\mathrm{d}t \\
\\
=\int_{Q_{1}^{\varepsilon }}g_{1}\left( x\right) q_{1}\left( t,x\right)
\hspace{0.03cm}\mathrm{d}t\mathrm{d}x+\int_{Q_{2}^{\varepsilon }}g_{2}\left(
x\right) q_{2}\left( t,x\right) \hspace{0.03cm}\mathrm{d}t\mathrm{d}x.%
\end{array}%
\end{equation*}

Passing to the limit in the last identity we get
\begin{equation*}
\begin{array}{l}
-\int_{Q_{1}}(\phi _{1}p_{1}\left( t,x\right) +\beta _{1}\left( \mathrm{div}%
\mathbf{u}\left( t,x\right) +\mathrm{div}_{y}\mathbf{u}_{1}\left(
t,x,y\right) \right) +\alpha _{1}\theta _{1}\left( t,x\right) )\partial
_{t}q_{1}\left( t,x\right) \hspace{0.03cm}\mathrm{d}t\mathrm{d}x\mathrm{d}y
\\
\\
-\int_{Q_{2}}(\phi _{2}p_{2}\left( t,x\right) +\beta _{2}\left( \mathrm{div}%
\mathbf{u}\left( t,x\right) +\mathrm{div}_{y}\mathbf{u}_{1}\left(
t,x,y\right) \right) +\alpha _{2}\theta _{2}\left( t,x\right) )\partial
_{t}q_{2}\left( t,x\right) \hspace{0.03cm}\mathrm{d}t\mathrm{d}x\mathrm{d}y
\\
\\
+\int_{\Omega \times Y_{1}}(\phi _{1}p_{1}\left( T,x\right) +\beta
_{1}\left( \mathrm{div}\mathbf{u}\left( T,x\right) +\mathrm{div}_{y}\mathbf{u%
}_{1}\left( T,x,y\right) \right) +\alpha _{1}\theta _{1}\left( T,x\right)
)q_{1}\left( T,x\right) \hspace{0.03cm}\mathrm{d}t\mathrm{d}x\mathrm{d}y \\
\\
+\int_{\Omega \times Y_{2}}(\phi _{2}p_{2}\left( T,x\right) +\beta
_{2}\left( \mathrm{div}\mathbf{u}\left( t,x\right) +\mathrm{div}_{y}\mathbf{u%
}_{1}\left( t,x,y\right) \right) +\alpha _{2}\theta _{2}\left( T,x\right)
)q_{2}\left( T,x\right) \hspace{0.03cm}\mathrm{d}x\mathrm{d}y \\
\\
-\int_{\Omega \times Y_{1}}(\phi _{1}p_{1}^{\varepsilon }\left( 0,x\right)
+\beta _{1}\left( \mathrm{div}\mathbf{u}\left( 0,x\right) +\mathrm{div}_{y}%
\mathbf{u}_{1}\left( 0,x,y\right) \right) +\alpha _{1}\theta _{1}\left(
0,x\right) )q_{1}\left( 0,x\right) \hspace{0.03cm}\mathrm{d}x\mathrm{d}y \\
\\
-\int_{\Omega \times Y_{2}}(\phi _{2}p_{2}\left( 0,x\right) +\beta
_{2}\left( \mathrm{div}\mathbf{u}\left( 0,x\right) +\mathrm{div}_{y}\mathbf{u%
}_{1}\left( 0,x,y\right) \right) +\alpha _{2}\theta _{2}\left( 0,x\right)
)q_{2}\left( 0,x\right) \hspace{0.03cm}\mathrm{d}x\mathrm{d}y \\
\\
+\int_{Q_{1}}\kappa _{1}\left( \nabla p_{1}\left( t,x\right) +\nabla _{y}%
\hat{p}_{1}\left( t,x,y\right) \right) \nabla q_{1}\left( t,x\right) \hspace{%
0.03cm}\mathrm{d}t\mathrm{d}x\mathrm{d}y \\
\\
+\int_{Q_{2}}\kappa _{2}\left( \nabla p_{2}\left( t,x\right) +\nabla _{y}%
\hat{p}_{2}\left( t,x,y\right) \right) \nabla q_{2}\left( t,x\right) \hspace{%
0.03cm}\mathrm{d}t\mathrm{d}x\mathrm{d}y \\
\\
+\int_{\Sigma _{T}}\varsigma \left( y\right) (p_{1}\left( t,x\right)
-p_{2}\left( t,x\right) )(q_{1}\left( t,x\right) -q_{2}\left( t,x\right) )%
\hspace{0.03cm}\mathrm{d}t\mathrm{d}x\mathrm{d}s(y) \\
\\
=\int_{Q_{1}\times Y}g_{1}\left( x\right) q_{1}\left( t,x\right) \hspace{%
0.03cm}\mathrm{d}x\mathrm{d}y\mathrm{d}t+\int_{Q_{2}\times Y}g_{2}\left(
x\right) q_{2}\left( t,x\right) \hspace{0.03cm}\mathrm{d}t\mathrm{d}x\mathrm{%
d}y.%
\end{array}%
\end{equation*}

Now, taking any sequence $\left( q_{m}^{n}\right) _{n}$ converging to $p_{m}$
in the last identity and integrating by parts with respect to $t$ yield (\ref%
{e89}). The identity (\ref{e90}) goes along the same lines as that of (\ref%
{e89}).
\end{proof}

\begin{proposition}
\label{p1}$e_{y}(\mathbf{\hat{u}})$ is an admissible test function and the
sequence $[e($\textbf{$u$}$^{\varepsilon })\left( x\right) -e($\textbf{$u$}$%
)\left( x\right) -e_{y}(\mathbf{\hat{u}})\left( x,\frac{x}{\varepsilon }%
\right) ]$ converges strongly to $0$ in $L^{2}\left( \Omega \right)
^{3\times 3}$.
\end{proposition}

\begin{proof}
We argue as in \cite{all}. From (\ref{e41}) -(\ref{e60}) and standard
results on regularity of elliptic equations \cite{eva}, $e_{y}(\mathbf{\hat{u%
}})$ is an admissible test function. Applying (\ref{e36}) yields for a.e. $%
t\in \left( 0,T\right) $
\begin{eqnarray}
&&\int_{\Omega }\left\{ \mathbf{A}\left( \frac{x}{\varepsilon }\right) [e(%
\mathbf{u}^{\varepsilon })\left( x\right) -e(\mathbf{u})\left( x\right)
-e_{y}(\mathbf{\hat{u}})\left( x,\frac{x}{\varepsilon }\right) ]\times
\right.  \notag \\
&&\left. [e(\mathbf{u}^{\varepsilon })\left( x\right) -e(\mathbf{u})\left(
x\right) -e_{y}(\mathbf{\hat{u}})\left( x,\frac{x}{\varepsilon }\right)
]\right\} \hspace{0.03cm}\mathrm{d}x=  \notag \\
&&-\int_{\Omega }\left( \mathbf{A}\mathbb{+}^{t}\mathbf{A}\right) e(\mathbf{u%
}^{\varepsilon })[e(\mathbf{u})+e_{y}(\mathbf{\hat{u}})\left( x,\frac{x}{%
\varepsilon }\right) ]\hspace{0.03cm}\mathrm{d}x  \notag \\
&&+\int_{\Omega }\mathbf{A}[e(\mathbf{u})+e_{y}(\mathbf{\hat{u}})\left( x,%
\frac{x}{\varepsilon }\right) ][e(\mathbf{u})+e_{y}(\mathbf{\hat{u}})\left(
x,\frac{x}{\varepsilon }\right) ]\hspace{0.03cm}\mathrm{d}x  \notag \\
&&\int_{\Omega }\mathbf{F}_{1}\mathbf{u}^{\varepsilon }\hspace{0.03cm}%
\mathrm{d}x-\int_{\Omega }\left( \beta _{1}p_{m}^{\varepsilon }+\gamma
_{1}\theta _{1}^{\varepsilon }\right) \mathrm{div}\mathbf{u}^{\varepsilon })%
\hspace{0.03cm}\mathrm{d}x  \notag \\
&&\int_{\Omega }\mathbf{F}_{2}\mathbf{u}^{\varepsilon }\hspace{0.03cm}%
\mathrm{d}x-\int_{\Omega }\left( \beta _{2}p_{2}^{\varepsilon }+\gamma
_{2}\theta _{2}^{\varepsilon }\right) \mathrm{div}\mathbf{u}^{\varepsilon })%
\hspace{0.03cm}\mathrm{d}x.  \label{e91}
\end{eqnarray}%
Using the strong convergences in $L^{2}\left( \Omega \right) $ of \textbf{$u$%
}$^{\varepsilon }$, $p_{m}^{\varepsilon }$ and $\theta _{m}^{\varepsilon }$
to \textbf{$u$}, $p_{m}$ and $\theta _{m}$ respectively and the weak
convergence of $\chi _{m}^{\varepsilon }$\textrm{$div$}$($\textbf{$u$}$%
^{\varepsilon })$ to $\int_{Y_{m}}\left( \mathrm{div}\mathbf{u+}\mathrm{div}%
_{y}\mathbf{\hat{u}}\right) $ and taking the limit in the last term of the
r.h.s. of (\ref{e91}) we obtain
\begin{eqnarray}
&&\lim_{\varepsilon \rightarrow 0}\int_{\Omega }\mathbf{F}_{m}\mathbf{u}%
^{\varepsilon }\hspace{0.03cm}\mathrm{d}x-\int_{\Omega }\left( \beta
_{m}p_{m}^{\varepsilon }+\gamma _{m}\theta _{m}^{\varepsilon }\right)
\mathrm{div}\mathbf{u}^{\varepsilon })\hspace{0.03cm}\mathrm{d}x  \notag \\
&=&\int_{\Omega }\mathbf{F}_{m}\mathbf{u}-\int_{\Omega }\left( \beta
_{m}p_{m}+\gamma _{m}\theta _{m}\right) \int_{Y_{m}}\left( \mathrm{div}%
\mathbf{u+}\mathrm{div}_{y}\mathbf{\hat{u}}\right) .  \label{e92}
\end{eqnarray}%
Now, as $\left( x,y\right) \longmapsto e_{y}(\mathbf{\hat{u}})\left(
x,y\right) $ is an admissible test function, the first two terms of the
right hand side of (\ref{e91}) converges to
\begin{eqnarray}
&&-\int_{\Omega \times Y}\left( a\mathbb{+}^{t}a\right) \left[ e(\mathbf{u}%
)+e_{y}(\mathbf{\hat{u}})\left( x,y\right) \right] [e(\mathbf{u})+e_{y}(%
\mathbf{\hat{u}})\left( x,y\right) ]\hspace{0.03cm}\mathrm{d}x\mathrm{d}y
\notag \\
&&+\int_{\Omega \times Y}a[e(\mathbf{u})+e_{y}(\mathbf{\hat{u}})\left(
x,y\right) ][e(\mathbf{u})+e_{y}(\mathbf{\hat{u}})\left( x,y\right) ]\hspace{%
0.03cm}\mathrm{d}x\mathrm{d}y  \notag \\
&=&-\int_{\Omega \times Y}a[e(\mathbf{u})+e_{y}(\mathbf{\hat{u}})\left(
x,y\right) ][e(\mathbf{u})+e_{y}(\mathbf{\hat{u}})\left( x,y\right) ]\hspace{%
0.03cm}\mathrm{d}x\mathrm{d}y.  \label{e93}
\end{eqnarray}%
Finally, thanks to the coercivity of $\mathbf{A}$, (\ref{e74}) (with $%
\mathbf{\hat{v}=\hat{u}}$) and (\ref{e92})-(\ref{e93}) we find that
\begin{equation*}
\begin{array}{l}
\alpha _{0}\left\Vert e(\mathbf{u}^{\varepsilon })\left( x\right) -e(\mathbf{%
u})\left( x\right) -e_{y}(\mathbf{\hat{u}})\left( x,\frac{x}{\varepsilon }%
\right) \right\Vert _{L^{2}\left( \Omega \right) ^{3\times 3}}\leq \\
\  \\
\int_{\Omega }\mathbf{A}\left( \frac{x}{\varepsilon }\right) [e(\mathbf{u}%
^{\varepsilon })\left( x\right) -e(\mathbf{u})\left( x\right) -e_{y}(\mathbf{%
\hat{u}})\left( x,\frac{x}{\varepsilon }\right) ][e(\mathbf{u}^{\varepsilon
})\left( x\right) -e(\mathbf{u})\left( x\right) -e_{y}(\mathbf{\hat{u}}%
)\left( x,\frac{x}{\varepsilon }\right) ]\rightarrow \\
\\
\sum_{m}\left\{ \int_{\Omega }\mathbf{F}_{m}\mathbf{u}-\int_{\Omega }\left(
\beta _{m}p_{m}+\gamma _{m}\theta _{m}\right) \int_{Y_{m}}\left( \mathrm{div}%
\mathbf{u+}\mathrm{div}_{y}\mathbf{\hat{u}}\right) \right\} \\
\\
-\int_{\Omega }\mathbf{A}[e(\mathbf{u})+e_{y}(\mathbf{\hat{u}})\left(
x,y\right) ][e(\mathbf{u})+e_{y}(\mathbf{\hat{u}})\left( x,y\right) ]=0%
\end{array}%
\end{equation*}%
where $\alpha _{0}=\min_{y\in \overline{Y}}\mathbf{A}\left( y\right) $.
Hence the proposition is proved.
\end{proof}

We now establish some corrector results on the mass conservation equation.
Let us first give some technical results.

\begin{lemma}[{G. Allaire and F. Murat \cite[Lemma A.4]{am}}]
There exists a constant $C>0$ such that for all $q_{1}\in H_{0}^{1}\left(
\Omega \right) \cap H^{1}\left( \Omega _{1}^{\varepsilon }\right) $ we have%
\begin{equation}
\left\Vert q_{1}\right\Vert _{0,\Omega _{1}^{\epsilon }}\leq C\left\Vert
\nabla q_{1}\right\Vert _{0,\Omega _{1}^{\epsilon }}.  \label{e94}
\end{equation}
\end{lemma}

\begin{lemma}
There exists a constant $C>0$ such that for all $q\in H^{1}\left( \Omega
_{1}^{\varepsilon }\right) $ we have
\begin{equation}
\varepsilon \left\Vert q\right\Vert _{0,\Sigma ^{\varepsilon }}^{2}\leq
C\left( \varepsilon ^{2}\left\Vert \nabla q\right\Vert _{0,\Omega
_{1}^{\varepsilon }}^{2}+\left\Vert q\right\Vert _{0,\Omega
_{1}^{\varepsilon }}^{2}\right) ,  \label{e96}
\end{equation}%
and%
\begin{equation}
\varepsilon \left\Vert q\right\Vert _{0,\Sigma ^{\varepsilon }}^{2}\leq
C\left( \left\Vert \nabla q\right\Vert _{0,\Omega _{1}^{\varepsilon
}}^{2}\right) .  \label{e97}
\end{equation}
\end{lemma}

\begin{proof}
We argue as in \cite{conc}. Using the trace theorem on $Y_{1}$ (see for e.g. R. A. Adams and J. F.
Fournier \cite{ada}), we know that there exists a constant $C\left(
Y_{1}\right) >0$ such that for every $\psi \in H^{1}\left( Y_{1}\right) $
\begin{equation*}
\int_{\Sigma }\left\vert \psi \right\vert ^{2}\hspace{0.03cm}\mathrm{d}%
\sigma \leq C\left( \int_{Y_{1}}\left\vert \nabla \psi \right\vert ^{2}%
\hspace{0.03cm}\mathrm{d}y+\int_{Y_{1}}\left\vert \psi \right\vert ^{2}%
\hspace{0.03cm}\mathrm{d}y\right) .
\end{equation*}%
Then, using the change of variables $y:=x/\varepsilon $ we have for every $%
q\in H^{1}\left( Y_{1}^{\epsilon k}\right) $
\begin{equation}
\varepsilon \int_{\Sigma ^{\epsilon k}}\left\vert q\right\vert ^{2}\hspace{%
0.03cm}\mathrm{d}\sigma ^{\varepsilon }\leq C\left( \varepsilon
^{2}\int_{Y_{1}^{\varepsilon k}}\left\vert \nabla q\right\vert ^{2}\hspace{%
0.03cm}\mathrm{d}x+\int_{Y_{1}^{\epsilon k}}\left\vert q\right\vert ^{2}%
\hspace{0.03cm}\mathrm{d}x\right)  \label{e98}
\end{equation}%
where $Y_{1}^{\epsilon k}=\varepsilon \left( k+\Omega _{1}^{\varepsilon
}\right) \ $and $\Sigma ^{\epsilon k}=\varepsilon \left( k+\Sigma \right) \
\ k\in \mathbb{Z}^{3}$. We mention that $C$ appearing in the inequality (\ref%
{e98}) is independent of $k\in \mathbb{Z}^{3}$. Summing up these
inequalities (\ref{e98}) over all $Y_{1}^{\epsilon k}$\ contained in $\Omega
$, we get (\ref{e91}). To obtain (\ref{e97}), it suffices to write (\ref{e96}%
)\ for sufficiently small $\varepsilon $, $\varepsilon \ll 1$ and use the
Friedrich inequality (\ref{e94}).
\end{proof}

\begin{lemma}
\label{l9} We have
\begin{eqnarray}
\lim_{\varepsilon \rightarrow 0}\int_{\Sigma _{T}^{\varepsilon }}\varepsilon
\varsigma (p_{m}^{\varepsilon }-p_{m})^{2}\hspace{0.03cm}\mathrm{d}%
s^{\varepsilon }(x) &=&0,  \label{e99} \\
\lim_{\varepsilon \rightarrow 0}\int_{\Sigma _{T}^{\varepsilon }}\varepsilon
\varsigma \left( p_{1}-p_{2}\right) ^{2}\hspace{0.03cm}\mathrm{d}%
s^{\varepsilon }(x) &=&\int_{Q}\widetilde{\zeta }\left( p_{1}-p_{2}\right)
^{2}\hspace{0.03cm}\mathrm{d}x.  \label{e100} \\
\lim_{\varepsilon \rightarrow 0}\int_{\Sigma _{T}^{\varepsilon }}\varepsilon
\varsigma (p_{1}^{\varepsilon }-p_{2}^{\varepsilon })^{2}\hspace{0.03cm}%
\mathrm{d}s^{\varepsilon }(x) &=&\int_{Q}\widetilde{\zeta }\left(
p_{1}-p_{2}\right) ^{2}\hspace{0.03cm}\mathrm{d}x.  \label{e101}
\end{eqnarray}
\end{lemma}

\begin{proof}
Using (\ref{e98}) yields
\begin{equation}
\varepsilon \left\Vert (p_{m}^{\varepsilon }-p_{m})\right\Vert _{0,\Sigma
^{\varepsilon }}^{2}\leq C\left( \varepsilon ^{2}\left\Vert \nabla
(p_{m}^{\varepsilon }-p_{m})\right\Vert _{0,\Omega _{m}^{\varepsilon
}}^{2}+\left\Vert p_{m}^{\varepsilon }-p_{m})\right\Vert _{0,\Omega
_{m}^{\varepsilon }}^{2}\right) .  \label{e102}
\end{equation}%
We know that $\left\Vert \nabla (p_{m}^{\varepsilon }-p_{m})\right\Vert
_{0,\Omega _{m}^{\varepsilon }}$ is uniformly bounded with respect to $%
\varepsilon $ and thanks to Rellich's Theorem, $\chi _{m}^{\varepsilon
}p_{m}^{\varepsilon }$ converges strongly to $\chi _{m}p_{m}$ in $%
L^{2}\left( \Omega \right) $. So, by passing to the limit in (\ref{e102}) we
easily arrive at (\ref{e99}). The convergence in (\ref{e100}) is a direct
consequence of Corollary \ref{c1}. To complete the proof it remains to show (%
\ref{e101}). Let us first observe that, according to Cauchy-Schwarz
inequality and (\ref{e99}), we have%
\begin{equation}
\left\vert \int_{\Sigma _{T}^{\varepsilon }}\varepsilon \varsigma
(p_{m}^{\varepsilon }-p_{m})\left( p_{1}-p_{2}\right) \hspace{0.03cm}\mathrm{%
d}s^{\varepsilon }(x)\right\vert \leq C\sqrt{\varepsilon }\left\Vert
p_{m}^{\varepsilon }-p_{m}\right\Vert _{0,\Sigma ^{\varepsilon }}\left\Vert
p_{1}-p_{2}\right\Vert _{0,\Sigma ^{\varepsilon }}\rightarrow 0\text{.}
\label{e103}
\end{equation}%
Hence, from (\ref{e99}), (\ref{e100}) and (\ref{e103}) we obtain
\begin{equation*}
\begin{array}{l}
\int_{0}^{T}\int_{\Sigma _{T}^{\varepsilon }}\varepsilon \varsigma
(p_{1}^{\varepsilon }-p_{2}^{\varepsilon })^{2}\hspace{0.03cm}\mathrm{d}%
s^{\varepsilon }(x)\mathrm{d}t=\int_{0}^{T}\int_{\Sigma ^{\varepsilon
}}\varepsilon \varsigma (p_{1}^{\varepsilon }-p_{1})^{2}\hspace{0.03cm}%
\mathrm{d}s^{\varepsilon }(x)\mathrm{d}t \\
\\
+\int_{\Sigma _{T}^{\varepsilon }}\varepsilon \varsigma (p_{2}^{\varepsilon
}-p_{2})^{2}\hspace{0.03cm}\mathrm{d}s^{\varepsilon }(x)\mathrm{d}%
t+2\int_{\Sigma _{T}^{\varepsilon }}\varepsilon \varsigma
(p_{1}^{\varepsilon }-p_{1})\left( p_{1}-p_{2}\right) \hspace{0.03cm}\mathrm{%
d}s^{\varepsilon }(x)\mathrm{d}t \\
\\
+2\int_{\Sigma _{T}^{\varepsilon }}\varepsilon \varsigma (p_{2}^{\varepsilon
}-p_{2})\left( p_{1}-p_{2}\right) \hspace{0.03cm}\mathrm{d}s^{\varepsilon
}(x)\mathrm{d}t \\
\\
+\int_{\Sigma _{T}^{\varepsilon }}\varepsilon \varsigma \left(
p_{1}-p_{2}\right) ^{2}\hspace{0.03cm}\mathrm{d}s^{\varepsilon }(x)\mathrm{d}%
s^{\varepsilon }(x)\mathrm{d}t \\
\\
\rightarrow \int_{Q}\widetilde{\zeta }\left( p_{1}-p_{2}\right) ^{2}\mathrm{d%
}x.%
\end{array}%
\end{equation*}%
Hence (\ref{e101}).
\end{proof}

\begin{proposition}
We have
\begin{equation*}
\lim_{\varepsilon \rightarrow 0}\left\Vert \chi _{m}\left( \frac{x}{%
\varepsilon }\right) \left\{ \nabla p_{m}^{\varepsilon }\left( t,x\right)
-\nabla p_{m}\left( t,x\right) -\nabla _{y}\hat{p}_{m}\left( t,x,\frac{x}{%
\varepsilon }\right) \right\} \right\Vert _{0,Q_{T}}=0.
\end{equation*}
\end{proposition}

\begin{proof}
As in the proof of Proposition \ref{p1}, we first write that%
\begin{eqnarray}
&&\int_{Q_{1}^{\varepsilon }}\mathcal{K}_{1}\left( \frac{x}{\varepsilon }%
\right) \left\vert \nabla p_{1}^{\varepsilon }\left( t,x\right) -\nabla
p_{1}\left( t,x\right) -\nabla _{y}\hat{p}_{1}\left( t,x,\frac{x}{%
\varepsilon }\right) \right\vert ^{2}\hspace{0.03cm}\mathrm{d}t\mathrm{d}x
\notag \\
&&+\int_{Q_{2}^{\varepsilon }}\mathcal{K}_{2}\frac{x}{\varepsilon }%
\left\vert \nabla p_{2}^{\varepsilon }\left( t,x\right) -\nabla p_{2}\left(
t,x\right) -\nabla _{y}\hat{p}_{2}\left( t,x,\frac{x}{\varepsilon }\right)
\right\vert ^{2}\hspace{0.03cm}\mathrm{d}t\mathrm{d}x  \notag \\
&=&\int_{Q_{1}^{\varepsilon }}\mathcal{K}_{1}\left( \frac{x}{\varepsilon }%
\right) \left\vert \nabla p_{1}^{\varepsilon }\left( t,x\right) \right\vert
^{2}\hspace{0.03cm}\mathrm{d}t\mathrm{d}x+\int_{Q_{1}^{\varepsilon }}%
\mathcal{K}_{m}\left( \frac{x}{\varepsilon }\right) \left\vert \nabla
p_{1}^{\varepsilon }\left( t,x\right) \right\vert ^{2}\hspace{0.03cm}\mathrm{%
d}t\mathrm{d}x  \notag \\
&&-2\int_{Q_{1}^{\varepsilon }}\mathcal{K}_{1}\left( \frac{x}{\varepsilon }%
\right) \nabla p_{1}^{\varepsilon }\left( t,x\right) \left( \nabla
p_{1}\left( t,x\right) +\nabla _{y}\hat{p}_{1}\left( t,x\right) \right)
\hspace{0.03cm}\mathrm{d}t\mathrm{d}x  \notag \\
&&-2\int_{Q_{2}^{\varepsilon }}\mathcal{K}_{2}\left( \frac{x}{\varepsilon }%
\right) \nabla p_{2}^{\varepsilon }\left( t,x\right) \left( \nabla
p_{2}\left( t,x\right) +\nabla _{y}\hat{p}_{2}\left( t,x,\frac{x}{%
\varepsilon }\right) \right) \hspace{0.03cm}\mathrm{d}t\mathrm{d}x  \notag \\
&&\int_{Q_{1}^{\varepsilon }}\mathcal{K}_{1}\left( \frac{x}{\varepsilon }%
\right) \left\vert \nabla p_{1}\left( t,x\right) +\nabla _{y}\hat{p}%
_{1}\left( t,x,\frac{x}{\varepsilon }\right) \right\vert ^{2}\hspace{0.03cm}%
\mathrm{d}t\mathrm{d}x  \notag \\
&&+\int_{Q_{2}^{\varepsilon }}\mathcal{K}_{2}\left( \frac{x}{\varepsilon }%
\right) \left\vert \nabla p_{2}\left( t,x\right) +\nabla _{y}\hat{p}%
_{2m}\left( t,x,\frac{x}{\varepsilon }\right) \right\vert ^{2}\hspace{0.03cm}%
\mathrm{d}t\mathrm{d}x.  \label{e104}
\end{eqnarray}%
It is easy to show that the last two terms of the right hand side of (\ref%
{e104}) converge to
\begin{eqnarray}
&&-2\int_{Q_{m}^{\varepsilon }}\mathcal{K}_{m}\left( \frac{x}{\varepsilon }%
\right) \nabla p_{m}^{\varepsilon }\left( t,x\right) \left( \nabla
p_{m}\left( t,x\right) +\nabla _{y}\hat{p}_{m}\left( t,x,\frac{x}{%
\varepsilon }\right) \right) +  \notag \\
&&\int_{Q_{m}^{\varepsilon }}\mathcal{K}_{m}\left( \frac{x}{\varepsilon }%
\right) \left\vert \nabla p_{m}\left( t,x\right) +\nabla _{y}\hat{p}%
_{m}\left( t,x,\frac{x}{\varepsilon }\right) \right\vert ^{2}  \notag \\
&\longrightarrow &  \notag \\
&&-\int_{Q\times Y_{m}}\mathcal{K}_{m}\left( y\right) \left\vert \nabla
p_{m}\left( t,x\right) +\nabla _{y}\hat{p}_{m}\left( t,x,y\right)
\right\vert ^{2},\ m=1,2.  \label{e105}
\end{eqnarray}%
Now by (\ref{e37}) one can see that
\begin{gather}
\int_{Q_{m}^{\varepsilon }}\mathcal{K}_{1}\left( \frac{x}{\varepsilon }%
\right) \nabla p_{1}^{\varepsilon }\nabla p_{1}^{\varepsilon }\hspace{0.03cm}%
\mathrm{d}t\mathrm{d}x+\int_{Q_{m}^{\varepsilon }}\mathcal{K}_{2}\left(
\frac{x}{\varepsilon }\right) \nabla p_{2}^{\varepsilon }\nabla
p_{2}^{\varepsilon }\hspace{0.03cm}\mathrm{d}t\mathrm{d}x  \notag \\
-\int_{Q_{1}^{\varepsilon }}\partial _{t}(\phi _{1}p_{1}^{\varepsilon
}+\beta _{1}\mathrm{div}\mathbf{u}_{1}^{\varepsilon }+\gamma _{1}\theta
_{1}^{\varepsilon })p_{1}^{\varepsilon }\hspace{0.03cm}\mathrm{d}t\mathrm{d}x
\notag \\
-\int_{Q_{m}^{\varepsilon }}\partial _{t}(\phi _{2}p_{2}^{\varepsilon
}+\beta _{2}\mathrm{div}\mathbf{u}_{2}^{\varepsilon }+\gamma _{2}\theta
_{2}^{\varepsilon })p_{2}^{\varepsilon }\hspace{0.03cm}\mathrm{d}t\mathrm{d}x
\notag \\
+\int_{\Sigma _{T}^{\varepsilon }}\varepsilon \varsigma (p_{1}^{\varepsilon
}-p_{2}^{\varepsilon })^{2}\mathrm{d}s^{\varepsilon }(x)\mathrm{d}%
t=\int_{Q_{m}^{\varepsilon }}G_{m}p_{m}^{\varepsilon }\hspace{0.03cm}\mathrm{%
d}t\mathrm{d}x.  \label{e106}
\end{gather}%
Taking the limit in the r.h.s. of (\ref{e106}) gives
\begin{equation}
\lim_{\varepsilon }\int_{Q_{m}^{\varepsilon }}G_{m}p_{m}^{\varepsilon }%
\hspace{0.03cm}\mathrm{d}t\mathrm{d}x=\left\vert Y_{m}\right\vert
\int_{Q}G_{m}p_{m}\hspace{0.03cm}\mathrm{d}t\mathrm{d}x,\ m=1,2.
\label{e107}
\end{equation}%
Using Lemma \ref{l9}, and more precisely (\ref{e101}) and passing in the
last term of the l.h.s. of (\ref{e106}) yields%
\begin{equation*}
\lim_{\varepsilon }\int_{\Sigma _{T}^{\varepsilon }}\varepsilon \varsigma
(p_{1}^{\varepsilon }-p_{2}^{\varepsilon })^{2}\hspace{0.03cm}\mathrm{d}%
s^{\varepsilon }(x)\mathrm{d}t=\int_{Q\times \Sigma }\varsigma \left(
y\right) \left( p_{1}-p_{2}\right) ^{2}\hspace{0.03cm}\mathrm{d}x\mathrm{d}%
s\left( y\right) \mathrm{d}t.
\end{equation*}%
Regarding the second term of the right hand side of (\ref{e106}) we proceed
as follows: Firstly, integrating par parts with respect to the time variable
$t$ we have for $m=1,2$%
\begin{eqnarray}
\int_{Q_{m}^{\varepsilon }}\partial _{t}(\phi _{m}p_{m}^{\varepsilon
})p_{m}^{\varepsilon }\hspace{0.03cm}\mathrm{d}t\mathrm{d}x
&=&\int_{Q_{m}^{\varepsilon }}\phi _{m}(p_{m}^{\varepsilon })^{2}\left(
T,x\right) \hspace{0.03cm}\mathrm{d}x  \notag \\
&\rightarrow &\int_{Q\times Y_{m}}\phi _{m}(p_{m})^{2}\left( T,x\right)
\hspace{0.03cm}\mathrm{d}t\mathrm{d}x.  \label{e108}
\end{eqnarray}%
Secondly, since $\chi _{m}^{\varepsilon }\partial _{t}($\textrm{$d$}$\mathrm{%
iv}$\textbf{$u$}$_{m}^{\varepsilon })$ converges weakly to $\chi _{m}\left\{
\partial _{t}(\mathrm{div}\mathbf{u}_{m}+\mathrm{div}_{y}\mathbf{\hat{u}}%
_{m})\right\} $ and $\chi _{m}^{\varepsilon }p_{m}^{\varepsilon }$ converges
strongly to $\chi _{m}p_{m}$ it follows that
\begin{eqnarray}
&&\lim_{\varepsilon }\int_{Q_{m}^{\varepsilon }}\beta _{m}\partial _{t}(%
\mathrm{div}\mathbf{u}_{m}^{\varepsilon })p_{m}^{\varepsilon }\hspace{0.03cm}%
\mathrm{d}t\mathrm{d}x  \notag \\
&=&\int_{Q\times Y_{m}}\beta _{m}\left\{ \partial _{t}(\mathrm{div}\mathbf{u}%
_{m}+\mathrm{div}_{y}\mathbf{\hat{u}}_{m})\right\} p_{m}\hspace{0.03cm}%
\mathrm{d}t\mathrm{d}x.  \label{art7-89}
\end{eqnarray}%
Furthermore, as $\chi _{m}^{\varepsilon }\partial _{t}\theta
_{m}^{\varepsilon }$ converges strongly to $\chi _{m}\partial _{t}\theta
_{m} $, we see that
\begin{equation}
\lim_{\varepsilon }\int_{Q_{m}^{\varepsilon }}\gamma _{m}\partial _{t}\theta
_{m}^{\varepsilon }p_{m}^{\varepsilon }\hspace{0.03cm}\mathrm{d}t\mathrm{d}%
x=\int_{Q\times Y_{m}}\gamma _{m}\partial _{t}\theta _{m}p_{m}\hspace{0.03cm}%
\mathrm{d}t\mathrm{d}x.  \label{e109}
\end{equation}%
Using the integral identity (\ref{e89}) and (\ref{e107})-(\ref{e109}) we
pass to the limit in (\ref{e106}) to get
\begin{eqnarray}
&&\lim_{\varepsilon }\left\{ \int_{Q_{1}^{\varepsilon }}\mathcal{K}%
_{m}\left( \frac{x}{\varepsilon }\right) \nabla p_{1}^{\varepsilon }\nabla
p_{1}^{\varepsilon }\hspace{0.03cm}\mathrm{d}t\mathrm{d}x+\int_{Q_{2}^{%
\varepsilon }}\mathcal{K}_{2}\left( \frac{x}{\varepsilon }\right) \nabla
p_{2}^{\varepsilon }\nabla p_{2}^{\varepsilon }\hspace{0.03cm}\mathrm{d}t%
\mathrm{d}x\right\}  \notag \\
&=&\left\vert Y_{1}\right\vert \int_{Q}G_{1}p_{1}\hspace{0.03cm}\mathrm{d}t%
\mathrm{d}x+\left\vert Y_{2}\right\vert \int_{\Omega }G_{2}p_{2}\hspace{%
0.03cm}\mathrm{d}t\mathrm{d}x  \notag \\
&&+\int_{\Omega \times Y_{1}}\phi _{1}(p_{1})^{2}\left( T,x\right) \hspace{%
0.03cm}\mathrm{d}x+\int_{\Omega \times Y_{2}}\phi _{2}(p_{2})^{2}\left(
T,x\right) \hspace{0.03cm}\mathrm{d}x+  \notag \\
&&\int_{Q\times Y_{1}}\left[ \partial _{t}(\beta _{1}\left( \mathrm{div}%
\mathbf{u}_{1}+\mathrm{div}_{y}\mathbf{\hat{u}}_{1}\right) +\gamma
_{1}\theta _{1}))\right] p_{1}\hspace{0.03cm}\mathrm{d}t\mathrm{d}x  \notag
\\
&&+\int_{Q\times Y_{m}}\left[ \partial _{t}(\beta _{2}\left( \mathrm{div}%
\mathbf{u}_{2}+\mathrm{div}_{y}\mathbf{\hat{u}}_{2}\right) +\gamma
_{2}\theta _{2}))\right] p_{2}\hspace{0.03cm}\mathrm{d}t\mathrm{d}x  \notag
\\
&&-\int_{Q\times \Sigma }\varsigma \left( y\right) \left( p_{1}-p_{2}\right)
^{2}\hspace{0.03cm}\mathrm{d}t\mathrm{d}x\mathrm{d}s\left( y\right).
\label{e110}
\end{eqnarray}%
Next, using (\ref{e80}), the convergence in (\ref{e110}) can be reduced to
\begin{eqnarray}
&&\lim_{\varepsilon }\left\{ \int_{Q_{1}^{\varepsilon }}\mathcal{K}%
_{1}\left( \frac{x}{\varepsilon }\right) \nabla p_{1}^{\varepsilon }\nabla
p_{1}^{\varepsilon }\hspace{0.03cm}\mathrm{d}t\mathrm{d}x+\int_{Q_{2}^{%
\varepsilon }}\mathcal{K}_{2}\left( \frac{x}{\varepsilon }\right) \nabla
p_{2}^{\varepsilon }\nabla p_{2}^{\varepsilon }\hspace{0.03cm}\mathrm{d}t%
\mathrm{d}x\right\} +  \notag \\
&=&\int_{Q\times Y_{1}}\mathcal{K}_{1}\left( y\right) (\nabla p_{1}+\nabla
_{y}\hat{p}_{1})(\nabla p_{1}+\nabla _{y}\hat{p}_{1}))\hspace{0.03cm}\mathrm{%
d}t\mathrm{d}x\mathrm{d}y+  \notag \\
&&\int_{Q\times Y_{2}}\mathcal{K}_{2}\left( y\right) (\nabla p_{2}+\nabla
_{y}\hat{p}_{2})(\nabla p_{2}+\nabla _{y}\hat{p}_{2}))\hspace{0.03cm}\mathrm{%
d}t\mathrm{d}x\mathrm{d}y.  \label{e111}
\end{eqnarray}%
Finally, collecting all the limits (\ref{e105}) and (\ref{e111}),
Equation (\ref{e104}) becomes
\begin{eqnarray*}
&&\lim_{\varepsilon }\sum_{m=1,2}\left\Vert \chi _{m}\left( \frac{x}{%
\varepsilon }\right) \left\{ \nabla p_{m}^{\varepsilon }\left( t,x\right)
-\nabla p_{m}\left( t,x\right) -\nabla _{y}\hat{p}_{m}\left( t,x,\frac{x}{%
\varepsilon }\right) \right\} \right\Vert _{0,Q}^{2}\leq \\
&&\lim_{\varepsilon }\sum_{m=1,2}\int_{Q_{m}^{\varepsilon }}\mathcal{K}%
_{m}\left\vert \nabla p_{m}^{\varepsilon }\left( t,x\right) -\nabla
p_{m}\left( t,x\right) -\nabla _{y}\hat{p}_{m}\left( t,x,\frac{x}{%
\varepsilon }\right) \right\vert ^{2}=0.
\end{eqnarray*}%
This gives the desired result.
\end{proof}

Now we state the corrector result for the heat equation whose proof follows
the same lines as that of Proposition \ref{p1}\ and is therefore omitted.

\begin{proposition} We have
\begin{equation*}
\lim_{\varepsilon \rightarrow 0}\left\Vert \chi _{m}\left( \frac{x}{%
\varepsilon }\right) \left\{ \nabla \theta _{m}^{\varepsilon }\left(
t,x\right) -\nabla \theta _{m}\left( t,x\right) -\nabla _{y}\theta
_{m}\left( t,x,\frac{x}{\varepsilon }\right) \right\} \right\Vert _{0,Q}=0.
\end{equation*}
\end{proposition}

Finally we state the asymptotic behavior of the energies.

\begin{proposition}
One has for a.e. $t\in \left( 0,T\right) $
\begin{eqnarray}
&&\lim_{\varepsilon }\left( \int_{\Omega }\mathbf{A}\left( \frac{x}{%
\varepsilon }\right) e(\mathbf{u}^{\varepsilon })e(\mathbf{u}^{\varepsilon })%
\hspace{0.03cm}\mathrm{d}x+\sum_{m=1,2}\int_{\Omega _{m}^{\varepsilon
}}\left( \beta _{m}\nabla p_{m}^{\varepsilon }+\gamma _{m}\nabla \theta
_{m}^{\varepsilon }\right) \mathbf{u}^{\varepsilon }\hspace{0.03cm}\mathrm{d}%
x\right)  \notag \\
&=&\int_{\Omega }a^{\hom }e(\mathbf{u})e(\mathbf{u})\hspace{0.03cm}\mathrm{d}%
x+\sum_{m=1,2}\int_{\Omega }\left( \mathcal{A}_{m}\nabla \theta _{m}+%
\mathcal{B}_{m}\nabla p_{m}\right) \mathbf{u}\hspace{0.03cm}\mathrm{d}x.
\label{e112}
\end{eqnarray}
\end{proposition}

\begin{proof}
Taking $\mathbf{v}=$\textbf{$u$}$^{\varepsilon }$ in (\ref{e36}) gives
\begin{eqnarray*}
&&\int_{\Omega }a\left( \frac{x}{\varepsilon }\right) e(\mathbf{u}%
^{\varepsilon })\left( x\right) e(\mathbf{u}^{\varepsilon })\hspace{0.03cm}%
\mathrm{d}x+\sum_{m=1,2}\int_{\Omega _{m}^{\varepsilon }}\left( \beta
_{m}\nabla p_{m}^{\varepsilon }+\beta _{m}\nabla \theta _{m}^{\varepsilon
}\right) \mathbf{u}^{\varepsilon }\hspace{0.03cm}\mathrm{d}x \\
&&-\int_{\Omega }\mathbf{Fu}^{\varepsilon }=0
\end{eqnarray*}%
which, by taking into account (\ref{e79}), tends to
\begin{eqnarray*}
&&\int_{\Omega }a^{\hom }e(\mathbf{u})e\left( \mathbf{u}\right) \hspace{%
0.03cm}\mathrm{d}x+\sum_{m=1,2}\int_{\Omega }\left( \mathcal{A}_{m}\nabla
\theta _{m}+\mathcal{B}_{m}\nabla p_{m}\right) \mathbf{u}\hspace{0.03cm}%
\mathrm{d}x \\
&&-\int_{\Omega }\mathbf{Fu}\hspace{0.03cm}\mathrm{d}x\text{=}0.
\end{eqnarray*}%
Hence (\ref{e112}).
\end{proof}

\section{Conclusion}

In this paper we derived by a homogenization technique a more general model
of thermoporoelasticity with double porosity and two temperatures. More
precisely, we studied a micro-model of fluid and thermal flows in two
component poroelastic media consisting of matrix and inclusions with the
same order of permeabilities and conductivities, separated by a periodic
and thin layer which forms an exchange fluid/thermal barrier. In particular,
we have shown that the Biot-Willis and thermal expansion parameters are in
that case matrices and no longer scalars, see for instance \cite{ain1, ain2}%
. Let us mention that the result of the paper remains valid if one considers
non homogeneous initial and/or Dirichlet \ conditions. An interesting
problem is to investigate the limiting behavior of such media when the flow
potential in the inclusions is rescaled by $\varepsilon ^{2}$. This occurs
especially when the flow in the inclusions presents very high frequency
spatial variations due to a relatively very low permeability, see Remark \ref%
{rem1}.

\noindent \textbf{Acknowledgment}
The author acknowledges the support of the Algerian ministry of higher
education and scientific research through the AMNEDP Laboratory.
\\

\end{document}